\documentclass[11pt]{amsart}
\usepackage{amsmath}
\usepackage{amsfonts}
\usepackage{amssymb}
\usepackage[all]{xy}
\usepackage{bbding}
\usepackage{txfonts}
\usepackage{amscd}
\usepackage{xspace}
\usepackage[shortlabels]{enumitem}
\usepackage{ifpdf}
\ifpdf
  \usepackage[colorlinks,final,backref=page,hyperindex]{hyperref}
\else
  \usepackage[colorlinks,final,backref=page,hyperindex]{hyperref}
\fi
\usepackage{tikz}

\newtheorem{thm}{Theorem}[section]
\newtheorem{lem}[thm]{Lemma}
\newtheorem{cor}[thm]{Corollary}
\newtheorem{pro}[thm]{Proposition}

\newtheorem{rmk}[thm]{Remark}
\newtheorem{defi}[thm]{Definition}

\newcommand {\emptycomment}[1]{}

\newcommand{\K}{\mathbb{K}}

\newcommand{\Integ}{\mathbb Z}

\newcommand{\huaB}{\mathcal{B}}

\newcommand{\huaH}{\mathcal{H}}

\newcommand{\huaO}{\mathcal{O}}

\newcommand{\huaZ}{\mathcal{Z}}

\newcommand{\g}{\mathfrak g}

\newcommand{\frkg}{\mathfrak g}

\newcommand{\frkT}{\mathfrak T}
\newcommand{\frkU}{\mathfrak U}

\newcommand{\frkX}{\mathfrak X}
\newcommand{\frkY}{\mathfrak Y}

\newcommand{\Courant}[1]{\left\llbracket  #1\right\rrbracket }

\newcommand{\Id}{\rm{Id}}

\newcommand{\br}[1]{   [ \cdot,    \cdot  ]   }
\newcommand{\id}{\mathrm{id}}

\newcommand{\dM}{\mathrm{d}}

\newcommand{\Hom}{\mathrm{Hom}}

\newcommand{\gl}{\mathfrak {gl}}

\newcommand{\nl}{n-\mathrm{Lie}}

\newcommand{\Ker}{\mathrm{Ker}}

\newcommand{\ad}{\mathrm{ad}}

\newcommand{\LP}{$\mathsf{LieRep}$~}

\allowdisplaybreaks
\begin{document}

\title{
{Lie $\MakeLowercase{n}$-algebras and cohomologies of relative Rota-Baxter operators on $\MakeLowercase{n}$-Lie algebras}
 }\vspace{2mm}
\author{Ming Chen}
\address{School of Mathematics and Statistics, Northeast Normal University, Changchun 130024, Jilin, China}
\email{chenm468@nenu.edu.cn}

\author{Jiefeng Liu}
\address{School of Mathematics and Statistics, Northeast Normal University, Changchun 130024, Jilin, China}
\email{liujf534@nenu.edu.cn}

\author{Yao Ma}
\address{School of Mathematics and Statistics, Northeast Normal University, Changchun 130024, Jilin, China}
\email{may703@nenu.edu.cn}
\vspace{-5mm}

\date{}
\renewcommand{\thefootnote}{\fnsymbol{footnote}}

\maketitle

\begin{abstract}
Based on the differential graded Lie algebra controlling deformations of an $n$-Lie algebra with a representation (called an $n$-\LP pair), we construct a Lie $n$-algebra, whose Maurer-Cartan elements  characterize relative Rota-Baxter operators on $n$-\LP pairs.  The notion of an $n$-pre-Lie algebra is introduced, which is the underlying algebraic structure of the relative Rota-Baxter operator.  We give the cohomology of relative Rota-Baxter operators and study infinitesimal deformations  and extensions of order $m$ deformations to order $m+1$ deformations of relative Rota-Baxter operators through the cohomology groups of relative Rota-Baxter operators. Moreover, we build the relation between the cohomology groups of relative Rota-Baxter operators on $n$-\LP pairs and those on $(n+1)$-\LP pairs by certain linear functions.
\end{abstract}

\tableofcontents
\section{Introduction}

In this paper, we use Maurer-Cartan elements of Lie $n$-algebras to characterize the relative Rota-Baxter operators on $n$-\LP pairs and study the cohomology and deformations of  relative Rota-Baxter operators on $n$-\LP pairs.

\subsection{$n$-Lie algebras,Lie $n$-algebras and relative Rota-Baxter operators on $n$-\LP pairs}
The notion of an $n$-Lie algebra, or a Filippov algebra was introduced in \cite{Fili85}. It is the algebraic structure corresponding to Nambu mechanics (\cite{Nambu,Tak}). $n$-Lie algebras, or more generally, $n$-Leibniz algebras, have attracted attention from several fields of mathematics and physics due to their close connection with dynamics, geometries as well as string theory (\cite{BaLa08,BaLa09,Casas02,CS,MeFi08,MeFi09,Fig08,Fig09,HHM,P}).  For example, the structure of $n$-Lie algebras are applied to the study of the supersymmetry and gauge symmetry transformations of the word-volume theory of multiple M2-branes  and the generalized identity for $n$-Lie algebras can be regarded as a generalized Pl${\rm\ddot{u}}$cker relation in the physics literature. See the review article \cite{review} for more details.

The notion of a Lie $n$-algebra was introduced by Hanlon and Wachs in \cite{HW}. Lie $n$-algebra is a special $L_\infty$-algebra, in which only the $n$-ary bracket is nonzereo. See \cite{Azc97,Sta85,Stash92} for more details on Lie $n$-algebras and  $L_\infty$-algebras. One useful method for constructing $L_\infty$-algebras is given by Vornov's higher derived bracket (\cite{Vor05}) and another one is by twisting with the Maurer-Cartan elements of a given $L_\infty$-algebra (\cite{Get09}). In this paper, we will use these methods to construct Lie $n$-algebras and $L_\infty$-algebras that characterize relative Rota-Baxter operators on $n$-Lie algebras as Maurer-Cartan elements and deformations of relative Rota-Baxter operators.

The classical Yang-Baxter equation plays a significant role in many fields in mathematics and mathematical physics such as quantum groups (\cite{CP1,Dr87}) and integrable systems (\cite{Semonov-Tian-Shansky}). In order to gain better understanding of the relationship between the classical Yang-Baxter equation and the related integrable systems, the more general notion of a relative Rota-Baxter operator (also called $\huaO$-operator) on a \LP pair was introduced by Kupershmidt (\cite{Kuper1}). To study solutions of $3$-Lie classical Yang-Baxter equation, the notion of a relative Rota-Baxter operator on a $3$-\LP pair was introduced in \cite{BaiC19}. A relative Rota-Baxter operator on a $3$-\LP pair $(\g;\ad)$, where $\ad$ is the adjoint representation of the $3$-Lie algebra $\g$, is exactly the Rota-Baxter operator on the $3$-Lie algebra $\g$ introduced in  \cite{BR}. See the book \cite{Guo} for more details and applications about Rota-Baxter operators.

In \cite{TBGS}, the authors showed that a relative Rota-Baxter operator on a Lie algebra is a Maurer-Cartan element of a graded Lie algebra. Recently, it was shown in \cite{THS}  that a relative Rota-Baxter operator on a $3$-Lie algebra is a Maurer-Cartan element of a Lie $3$-algebra. The first purpose in this paper is to realize the relative Rota-Baxter operators on $n$-Lie algebras as Maurer-Cartan elements of Lie $n$-algebras.

Pre-Lie algebras are a class of nonassociative algebras coming from the study of convex
homogeneous cones, affine manifolds and affine structures on Lie groups, and the
cohomologies of associative algebras. See the survey \cite{Pre-lie algebra in geometry} and the references therein for more details. The beauty of a pre-Lie algebra is that the commutator gives rise to a Lie algebra and the
left multiplication gives rise to a representation of the commutator Lie algebra. Conversely,  a relative Rota-Baxter operator action on a Lie algebra gives rise to a pre-Lie algebra (\cite{GS,Kuper1}) and thus the pre-Lie algebra can be seen as the underlying algebraic structure of a relative Rota-Baxter operator. In this paper, we introduce the notion of an $n$-pre-Lie algebra, which gives an $n$-Lie algebra naturally and its left multiplication operator gives rise to a representation of this $n$-Lie algebra. An $n$-pre-Lie algebra can also be obtained through the action of a relative Rota-Baxter operator on an $n$-Lie algebra.
\subsection{Deformations and cohomology theories}
The theory of deformation plays a prominent role in mathematics and physics. In physics, the ideal of deformation appears in the perturbative quantum field theory and quantizing classical mechanics. The idea of treating deformation as a tool to study the algebraic structures was introduced by Gerstenhaber in his work of associative algebras (\cite{Gerstenhaber1,Gerstenhaber2}) and then was extended to Lie algebras by Nijenhuis and Richardson (\cite{NR1,NR2}).  Deformations of $3$-Lie algebras and $n$-Lie algebras are studied in \cite{Arfa18,Fig09,LiuJ16,Ta}. See the review paper \cite{review, Makhlouf} for more details. Recently,  people pay more attention on the deformations of morphisms (\cite{Arfa18,Fre04,Fre15,Fre152,Man,Yau07}), relative Rota-Baxter operators (\cite{Das20,TBGS,THS}) and diagram of algebras (\cite{Bar,GS,Markl}).

Usually the cohomology theory is an important tool in the study of deformations of a mathematical structure.  Typically, infinitesimal deformations are classified by a suitable second cohomology group and the obstruction of an order $n$ deformation extending to an order $n+1$ deformation can be controlled by the third cohomology group. Cohomology and deformations of relative Rota-Baxter operators ($\huaO$-operators) on associative algebras, Lie algebras and $3$-Lie algebras were studied in \cite{Das20,TBGS,THS}.

In the present paper, we study the cohomology theory of relative Rota-Baxter operators on $n$-\LP pairs. By using the underlying $n$-Lie algebra of a relative Rota-Baxter operator, we introduce the cohomology of a relative Rota-Baxter operator on an $n$-\LP pair. Then we study infinitesimal deformations and order $n$ deformations of a relative Rota-Baxter operator on an $n$-\LP pair. We show that infinitesimal deformations of relative Rota-Baxter operators are classified by the first cohomology group and that a higher order deformation of a relative Rota-Baxter operator is extendable if and only if its obstruction class in the second cohomology group of the relative Rota-Baxter operator is trivial.

\subsection{Outline of the paper}
In Section \ref{sec:MC-characterizations LP pairs}, we first recall representations and cohomologies
of $n$-Lie algebras and then construct a graded Lie algebra whose Maurer-Cartan elements are precisely $n$-\LP pairs.  In Section \ref{sec:MC-characterization RB-operators}, we use Voronov's higher derived brackets to construct a Lie $n$-algebra from an $n$-\LP pair and show that relative Rota-Baxter operators on $n$-\LP pairs can be characterized by Maurer-Cartan elements of the constructed Lie $n$-algebra. We give the notion of an $n$-pre-Lie algebra and show that a relative Rota-Baxter operator on an $n$-\LP pair can give rise to an $n$-pre-Lie algebra naturally. In Section \ref{sec:cohomology}, we define the cohomology of
a relative Rota-Baxter operator on an $n$-\LP pair using a certain $n$-\LP pair constructed by the relative Rota-Baxter operator. In Section \ref{sec:deformation}, we use the cohomology theory of relative Rota-Baxter operators to study deformations of relative Rota-Baxter operators. In Section \ref{sec:construction}, we study the relation between the cohomology groups of relative Rota-Baxter operators on $n$-\LP pairs and those on $(n+1)$-\LP pairs by certain linear functions.

In this paper, all the vector spaces are over an algebraically closed field $\mathbb K$ of characteristic $0$, and finite dimensional.

\vspace{2mm}

\section{Maurer-Cartan characterizations of $n$-\LP pairs}\label{sec:MC-characterizations LP pairs}
\subsection{Representations and cohomology of $n$-Lie algebras}
\begin{defi}{\rm(\cite{Fili85})}
An {\bf $n$-Lie algebra} is a vector space $\g$ together with a skew-symmetric linear map $[-,\cdots,-]_{\frak g}:\wedge^n{\frak g}\rightarrow {\frak g}$ such that for $x_i, y_j\in {\frak g}, 1\leq i\leq n-1,1\leq j\leq n$, the following identity holds:
\begin{equation}\label{1}
  [x_1,\cdots,x_{n-1},[y_1,\cdots,y_n]_{\frak g}]_{\frak g}=\sum\limits_{i=1}^n[y_1,\cdots,y_{i-1},[x_1,\cdots,x_{n-1},y_i]_{\frak g},y_{i+1},\cdots,y_n]_{\frak g}.
\end{equation}

 \end{defi}
For $x_1,x_2,\cdots,x_{n-1}\in \frak g$, define $\ad_{x_1,x_2,\cdots,x_{n-1}}\in \frak \gl(\frak g)$ by $\ad_{x_1,x_2,\cdots,x_{n-1}}y:=[x_1,x_2,\cdots,x_{n-1},y]_\g,  \forall  y \in \frak g.$
Then $\ad_{x_1,x_2,\cdots,x_{n-1}}$ is a derivation, i.e. $$\ad_{x_1,x_2,\cdots,x_{n-1}}[y_1,\cdots,y_n]_{\frak g}=\sum\limits_{i=1}^n[y_1,\cdots,y_{i-1},\ad_{x_1,x_2,\cdots,x_{n-1}}y_i,y_{i+1},\cdots,y_n]_{\frak g}.$$

\begin{defi}{\rm(\cite{Kasy87})}\label{defi:rep}
A {\bf representation} of an $n$-Lie algebra $(\frak g,[-,\cdots,-]_{\frak g})$ on a vector space $V$ is a linear map: $\rho:\wedge^{n-1}{\frak g}\rightarrow {\frak gl(V)},$ such that for all $x_1, x_2,\cdots,x_{n-1},y_1,\cdots,y_n\in \frak g,$ the following equalities hold:
\begin{eqnarray}
\label{eq:rep1}[\rho(\mathfrak{X}),\rho(\frkY)]&=&\rho(\mathfrak{X}\circ \frkY),\\
\label{eq:rep2}\rho(x_1,\cdots,x_{n-2},[y_1,\cdots,y_n])&=&\sum\limits_{i=1}^n(-1)^{n-i}\rho(y_1,\cdots,\widehat{y_{i}},\cdots,y_n)
\rho(x_1,\cdots,x_{n-2},y_i),
\end{eqnarray}
where $\mathfrak{X}=x_1\wedge\cdots\wedge x_{n-1},
\frkY=y_1\wedge\cdots\wedge y_{n-1}$ and  $\mathfrak{X}\circ \frkY=\sum\limits_{i=1}^{n-1}(y_1,\cdots,y_{i-1},[x_1,\cdots,x_{n-1},y_i]_\g,y_{i+1},\cdots,y_{n-1})$.
\end{defi}

Let $(V;\rho)$ be a representation of an $n$-Lie algebra $(\frak g,[-,\cdots,-]_{\frak g}).$  Denote by $$C^{m}_{\nl}(\g;V)=\Hom(\otimes^{m-1}(\wedge^{n-1}\g)\wedge\frak{g},V),\quad (m\geq 1),$$
which is the space of $n$-cochains. Define $\partial_\rho:C^{m}_{\nl}(\g;V)\rightarrow C^{m+1}_{\nl}(\g;V)$ by
\begin{eqnarray*}
\nonumber(\partial_\rho f)(\mathfrak{X}_1,\cdots,\mathfrak{X}_m,x_{m+1})&=&\sum\limits_{1\leq j< k\leq m}(-1)^jf(\mathfrak{X}_1,\cdots,\widehat{\mathfrak{X}_j},\cdots,\mathfrak{X}_{k-1},\mathfrak{X}_j\circ \mathfrak{X}_k,\mathfrak{X}_{k+1},\cdots,\mathfrak{X}_{m},x_{m+1})\\
&&+\sum\limits_{j=1}^m(-1)^{j}f(\mathfrak{X}_1,\cdots,\widehat{\mathfrak{X}_j},\cdots,\mathfrak{X}_{m},
[\mathfrak X_{j},x_{m+1}]_\g)\\
&&+\sum\limits_{j=1}^m(-1)^{j+1}\rho(\mathfrak X_{j})f(\mathfrak{X}_1,\cdots,\widehat{\mathfrak{X}_j},\cdots,\mathfrak{X}_{m},
x_{m+1})\\
&&+\sum\limits_{i=1}^{n-1}(-1)^{n+m-i+1}\rho(x^1_m,\cdots,\widehat{x^i_m},\cdots,x^{n-1}_m,x_{m+1})f(\mathfrak{X}_1,\cdots,\mathfrak{X}_{m-1},
x^i_{m})
\end{eqnarray*}
for any $\mathfrak{X}_i=x_i^1\wedge\cdots\wedge x_i^{n-1}\in \wedge^{n-1}\frak \g,i=1,2,\cdots,m,x_{m+1}\in \g.$

It was proved in \cite{Casas02} that $\partial_\rho\circ \partial_\rho=0.$ Thus, $(\oplus^{+\infty}_{m=1}C^{m}_{\nl}(\g;V),\partial_\rho)$ is a cochain complex.

\begin{defi}
The cohomology of the $n$-Lie algebra $\mathfrak{g}$ with coefficients in $V$ is the cohomology of the cochain complex  $(\oplus^{+\infty}_{m=1}C^{m}_{\nl}(\g;V),\partial_\rho)$.
 The corresponding $n$-th cohomology group is denoted by $H^m_{\nl}(\g;V)$.
 \end{defi}

\subsection{Bidegrees and graded Lie algebras on $C^{*}_{\nl}(\frak g;\frak g)$ }
Let $\g_1$ and $\g_2$ be two vector spaces. Denote by $\g^{l,k}$ the subspace of
$\otimes^{m}(\wedge^{n-1}(\g_1\oplus \g_2))\wedge(\g_1\oplus \g_2)$
which contains the number of $\g_1$ (resp. $\g_2$) is $l$ (resp. $k$).

The vector space
$\otimes^m(\wedge^{n-1}(\g_1\oplus \g_2))\wedge(\g_1\oplus \g_2)$ can be extended into the direct sum of $\g^{l,k},l+k=m(n-1)+1$. Furthermore, we have the following isomorphism
\begin{equation}\label{1234qwer}
C^m(\g_1\oplus \g_2,\g_1\oplus \g_2)\cong\sum\limits_{l+k=m(n-1)+1}\Hom(\g^{l,k},\g_1)\oplus\sum\limits_{l+k=m(n-1)+1}\Hom(\g^{l,k},\g_2).
\end{equation}
An element $f\in \Hom(\g^{l,k},\g_1)$(resp. $f\in \Hom(\g^{l,k},\g_2)$) naturally gives an element $\hat{f}\in C^m(\g_1\oplus \g_2,\g_1\oplus \g_2)$, which is called its lift. For example, the lifts of linear maps $\mu: \wedge^n\g_1\rightarrow\g_1, \rho: \wedge^{n-1}\g_1\otimes\g_2\rightarrow\g_2$ are defined by
\begin{eqnarray}
\hat{\mu}((x_1,u_1),\cdots,(x_n,u_n))&=& (\mu(x_1,\cdots, x_n),0),\\
\hat{\rho}((x_1,u_1),\cdots,(x_n,u_n))&=&(0,\sum_{i=1}^n(-1)^{n-i}\rho(x_1,\cdots,\hat{x_i},\cdots,x_n)(u_i)),
\end{eqnarray}
respectively. Let $H:\g_2\rightarrow \g_1$ be a linear map. Its lift is given by $\hat{H}(x,u)=(H(u),0)$.

\begin{defi}\label{bidegree-defi}
A linear map $f\in \Hom(\otimes^m(\wedge^{n-1}(\g_1\oplus \g_2))\wedge(\g_1\oplus \g_2),\g_1\oplus \g_2)$ has a {\bf bidegree} $k|l$, which
is denoted by $\|f\|=k|l$, if f satisfies the following four conditions:
\begin{enumerate}
  \item $k+l=m(n-1)$
  \item If $X$ is an element in $\g^{k+1,l}$, then $f(X)\in \g_1$
  \item If $X$ is an element in $\g^{k,l+1}$, then $f(X)\in \g_2$
  \item All the other case, $f(X)=0$.
\end{enumerate}
\end{defi}
We denote the set of linear maps of bidegree $k|l$ by $C^{k\mid l}(\g_1\oplus \g_2,\g_1\oplus \g_2)$. A linear map $f$ is said to be homogeneous if $f$ has a bidegree.

\vspace{3mm}

Let $\frak g$ be a vector space. We consider the graded vector space
$$C^{*}_{\nl}(\frak g;\frak g)=\oplus_{m\geq 0}C^{m}_{\nl}(\frak g;\frak g)=\oplus_{m\geq 0} \Hom(\otimes^m(\wedge^{n-1}\g)\wedge\frak{g},\g).$$
Then the graded vector space $C^{*}_{\nl}(\frak g;\frak g)$ equipped with the graded commutator bracket
\begin{equation}\label{14}
  [P,Q]_{\nl}=P\circ Q-(-1)^{pq}Q\circ P, \qquad\forall ~P\in C^p(\frak g,\frak g),Q\in C^q(\g,\g)
\end{equation}
is a graded Lie algebra (\cite{Rot05}), where $P\circ Q\in C^{p+q}(\frak g,\frak g)$ is defined by
\begin{eqnarray*}
&&(P\circ Q)(\mathfrak{X}_1,\cdots,\mathfrak{X}_{p+q},x)\\
&=&\sum\limits^{p}_{k=1}(-1)^{(k-1)q}\sum\limits_{\sigma\in S(k-1,q)}(-1)^{\sigma}\sum\limits^{n-1}_{i=1}P(\mathfrak{X}_{\sigma(1)},\cdots,\mathfrak{X}_{\sigma(k-1)},x^1_{k+q}\wedge\cdots\wedge x^{i-1}_{k+q}\wedge\\
&&Q(\mathfrak{X}_{\sigma(k)},\cdots,\mathfrak{X}_{\sigma(k+q-1)},x^i_{k+q})\wedge x^{i+1}_{k+q}\wedge\cdots\wedge x^{n-1}_{k+q},\mathfrak{X}_{k+q+1},\cdots,
\mathfrak{X}_{p+q},x)\\
&& +\sum\limits_{\sigma\in S(p,q)}(-1)^{pq}(-1)^{\sigma}P(\mathfrak{X}_{\sigma(1)},\cdots,\mathfrak{X}_{\sigma(p)},Q(\mathfrak{X}_{\sigma(p+1)},\cdots,\mathfrak{X}_{\sigma(p+q-1)},
\mathfrak{X}_{\sigma(p+q)},x)),
\end{eqnarray*}
for all $\mathfrak{X}_i=x^1_i\wedge\cdots\wedge x^{n-1}_i\in\wedge^{n-1}\frak g, i=1,2,\cdots,p+q$ and $x\in \g$.  In particular, $\mu:\wedge^n\g\rightarrow \g$ defines an $n$-Lie structure on $\g$ if and only if $[\mu,\mu]_{\nl}=0$, i.e. $\mu$ is a Maurer-Cartan element of
the graded Lie algebra $(C^{*}_{\nl}(\frak g;\frak g),[-,-]_{\nl})$. Moreover, the coboundary operator $\partial$ of the $n$-Lie algebra with the coefficients in the adjoint representation can be given by
$$\partial_{\ad}  f=(-1)^{p}[\mu ,f]_{\nl}, \quad \forall f \in C^p(\g,\g).$$

The following lemma shows that the graded Lie algebra structure on $C^{*}_{\nl}(\frak g;\frak g)$ is compatible with the bigrading.
\begin{lem}\label{thm:bidegree}
Let $f\in C^p(\g_1\oplus \g_2,\g_1\oplus \g_2)$ and $g\in C^q(\g_1\oplus \g_2,\g_1\oplus \g_2)$ be the homogeneous linear maps with bidegrees $k_f|l_f$ and $k_g|l_g$ respectively. Then the graded Lie bracket $[f,g]_{\nl}\in C^{p+q}(\g_1\oplus \g_2,\g_1\oplus \g_2)$ is a homogeneous linear map of bidegree $(k_f+k_g)|(l_f+l_g)$.
\end{lem}
\begin{proof}
  It follows by a direct calculation.
\end{proof}

\subsection{Maurer-Cartan characterizations of $n$-\LP pairs}
\begin{defi}
An {\bf $n$-\LP pair} is a pair of an $n$-Lie algebra $(\g,[-,\cdots,-]_\g)$ together with a representation $\rho$ on $V$. We denote an $n$-\LP pair by $(\g,[-,\cdots,-]_{\g};\rho)$, or simply $(\g;\rho)$.
\end{defi}
For $2$-\LP pair, we usually call it  \LP pair for short.

Let $(\frak g,[-,\cdots,-]_{\frak g};\rho)$ be an $n$-LieRep pair. Usually we will also use $\mu$ to indicate the $n$-Lie bracket $[-,\cdots,-]_{\frak g}$. Then $\mu+\rho$ corresponds to the semidirect product $n$-Lie algebra structure on $\g\oplus V$ given by
\begin{equation}\label{15}
  [x_1+u_1,\cdots,x_n+u_n]_{\rho}=[x_1,\cdots,x_n]_{\frak g}+\sum\limits^{n}_{i=1}(-1)^{n-i}\rho(x_1,\cdots,\widehat{x_i},\cdots,x_n)u_i.
\end{equation}
 We denote the semidirect product $n$-Lie algebra  by $\g\ltimes_{\rho} V$. Since $\mu\in \Hom(\wedge^n \g,\g)$ and $\rho\in \Hom(\wedge^{n-1} \g\otimes V,V)$, we have $\|\mu\|=n-1|0$ and $\|\rho\|=n-1|0$. Thus $\mu+\rho\in C^{n-1\mid 0}_{\nl}(\g\oplus V,\g\oplus V)$.

\begin{pro}
  Let $\g$ and $V$ be two vector spaces. Then $(\oplus_{k=0}^{+\infty}C^{k(n-1)\mid 0}(\g\oplus V,\g\oplus V),[-,-]_{\nl})$ is a graded Lie algebra. Its Maurer-Cartan element are precisely $n$-\LP pairs.
\end{pro}
\begin{proof}
 For $f\in C^{k(n-1)\mid 0}(\g_1\oplus \g_2,\g_1\oplus \g_2)$ and $g\in C^{l(n-1)\mid 0}(\g_1\oplus \g_2,\g_1\oplus \g_2)$, by Lemma \ref{thm:bidegree}, we have $||[f,\g]_{\nl}||=(k+l)(n-1)| 0$. Thus $(\oplus_{k=0}^{+\infty}C^{k(n-1)\mid 0}(\g\oplus V,\g\oplus V),[-,-]_{\nl})$ is a graded Lie subalgebra of $(C^*(\g\oplus V,\g\oplus V),[-,-]_{\nl})$.

 By a direct calculation, we have
\begin{eqnarray*}
&&\frac{1}{2}[\mu+\rho,\mu+\rho]_{\nl}\big{(x_1+u_1,\cdots,x_{n-1}+u_{n-1},y_1+v_1,\cdots,y_{n-1}+v_{n-1},y_n+v_n)}\\
&=&\sum\limits_{i=1}^n(\mu+\rho)(y_1+v_1,\cdots,y_{i-1}+v_{i-1},(\mu+\rho)(x_1+u_1,\cdots,x_{n-1}+u_{n-1},y_i+v_i),y_{i+1}+v_{i+1},\cdots,y_n+v_n)\\
&&-(\mu+\rho)\big{(x_1+u_1,\cdots,x_{n-1}+u_{n-1},(\mu+\rho)(y_1+v_1,\cdots,y_{n-1}+v_{n-1},y_n+v_n)})\\
&=&\sum\limits_{i=1}^n[y_1,\cdots,y_{i-1},[x_1,\cdots,x_{n-1},y_i]_\g,y_{i+1},\cdots,y_n]_\g-[x_1,\cdots,x_{n-1},[y_1,\cdots,y_{n-1},y_n]_\g]_\g\\
&&+\sum\limits_{k=1}^n(-1)^{n-k}\rho(y_1,\cdots,\hat{y_k},\cdots,y_n)
\rho(x_1,\cdots,x_{n-1})(v_k)-\sum\limits_{k=1}^n(-1)^{n-k}\rho(x_1,\cdots,x_{n-1})\rho(y_1,\cdots,\hat{y_k},\cdots,y_n)(v_k)\\
&&+\sum\limits_{i=1}^n\sum\limits_{k\neq i,k=1}^n(-1)^{n-k}\rho(y_1,\cdots,y_{i-1},[x_1,\cdots,x_{n-1},y_i]_\g,\cdots,\hat{y_k},\cdots,y_n)(v_k)\\
&&+\sum\limits_{i=1}^n(-1)^{n-i}\sum\limits_{j=1}^{n-1}(-1)^{n-j}\rho(y_1,\cdots,\hat{y_i},\cdots,y_n)\rho(x_1,\cdots,\hat{x_j},\cdots,x_{n-1},y_i)(u_j)\\
&&-\sum\limits_{j=1}^{n-1}(-1)^{n-j}\rho(x_1,\cdots,\hat{x_j},\cdots,x_{n-1},[y_1,\cdots,y_n]_\g)(u_j).
\end{eqnarray*}
Thus $[\mu+\rho,\mu+\rho]=0$ if and only if $\mu$ defines an $n$-Lie algebra structure and $\rho$ is a representation of $(\g,\mu)$ on $V$.
\end{proof}

Let $(\g,\mu;\rho)$ be an $n$-\LP pair. Note that $\pi=\mu+\rho$ is the Maurer-Cartan element of the graded Lie algebra $(\oplus_{k=0}^{+\infty}C^{k(n-1)\mid 0}(\g\oplus V,\g\oplus V),[-,-]_{\nl})$, by the graded Jacobi identity,  $\dM_\pi:=[\pi,-]_{\nl}$ is a graded derivation of the graded Lie algebra $(\oplus_{k=0}^{+\infty}C^{k(n-1)\mid 0}(\g\oplus V,\g\oplus V),[-,-]_{\nl})$ and satisfies $\dM_\pi^2=0$. Thus we have
\begin{pro}
  Let $(\g,\mu;\rho)$ be an $n$-\LP pair. Then $(\oplus_{k=0}^{+\infty}C^{k(n-1)\mid 0}(\g\oplus V,\g\oplus V),[-,-]_{\nl},\dM_\pi)$ is a differential graded Lie algebra. Furthermore, $(\g,\mu+\mu';\rho+\rho')$ is also an $n$-\LP pair for $\mu'\in \Hom(\wedge^n \g,\g)$ and $\rho'\in \Hom(\wedge^{n-1} \g\otimes V,V)$ if and only if $\mu'+\rho'$ is a Maurer-Cartan element of the differential graded Lie algebra $(\oplus_{k=0}^{+\infty}C^{k(n-1)\mid 0}(\g\oplus V,\g\oplus V),[-,-]_{\nl},\dM_\pi)$.
\end{pro}
\begin{proof}
  It follows by a direct calculation.
\end{proof}

\section{Maurer-Cartan characterization of relative Rota-Baxter operators on $n$-\LP pairs}\label{sec:MC-characterization RB-operators}
In this section, we apply higher derived brackets introduced by Voronov in \cite{Vor05} to construct the Lie $n$-algebra that characterizes relative Rota-Baxter operators as Maurer-Cartan elements.
\subsection{$L_{\infty}$-algebras, Lie $n$-algebras and higher derived brackets}

A permutation $\sigma\in S_n$ is called an $(i,n-i)$-shuffle if $\sigma(1)<\cdots<\sigma(i)$ and $\sigma(i+1)<\cdots<\sigma(n)$. If $i=0$ or $n$ we assume $\sigma=id$. The set of all $(i,n-i)$-shuffles will be denoted by $S(i,n-i)$.
\begin{defi}{\rm(\cite{Stash92})}
An $L_{\infty}$-algebra is a $\Integ$-graded vector space $\frak g=\oplus_{k\in Z}\frak g^k$ equipped with a collection $(k\geq1)$ of linear maps $l_k:\otimes^k\frak g\rightarrow\frak g$ of degree 1 with the property that, for any  homogeneous elements $x_1,\cdots,x_n\in \frak g$, we have
\begin{enumerate}
\item[\rm(i)] for every $\sigma\in S_n$,
$$l_n(x_{\sigma(1)},\cdots,x_{\sigma(n-1)},x_{\sigma(n)})=\varepsilon(\sigma)l_n(x_1,\cdots,x_{n-1},x_n),$$
\item[\rm(ii)]for all $n\geq 1$,
\begin{equation}\label{eq:general-JI}
\sum\limits_{i=1}^n\sum\limits_{\sigma\in S(i,n-i)}\varepsilon(\sigma)l_{n-i+1}(l_i(x_{\sigma(1)},\cdots,x_{\sigma(i)}),x_{\sigma(i+1)},\cdots,x_{\sigma(n)})=0.
\end{equation}
\end{enumerate}
\end{defi}

The notion of a Lie $n$-algebra was introduced in \cite{HW95}. A Lie $n$-algebra is a special $L_{\infty}$-algebra, in which only $n$-ary bracket is nonzero.
\begin{defi}
A Lie $n$-algebra is a $\Integ$-graded vector space $\frak g=\oplus_{k\in \Integ}\frak g^k$ equipped with an $n$-multilinear bracket $\{-,\cdots,-\}$ of degree $1$ satisfying,
\begin{enumerate}
  \item[\rm(i)]for all homogeneous elements $x_1,\cdots,x_n\in \frak g,$
  \begin{equation}\label{6}
{\{x_1,x_2,\cdots,x_n\}_{\frak g}}=\varepsilon(\sigma)\{x_{\sigma(1)},\cdots,x_{\sigma(n-1)},x_{\sigma(n)}\}_{\frak g},
\end{equation}
  \item[\rm(ii)] for all homogeneous elements $x_i\in\frak g, 1\leq i\leq 2n-1$
  \begin{equation}\label{7}\sum\limits_{\sigma\in S(n,n-1)}\varepsilon(\sigma)\{\{x_{\sigma(1)},\cdots,x_{\sigma(n)}\}_{\frak g},x_{\sigma(n+1)},\cdots,x_{\sigma(2n-1)}\}_{\frak g}=0.
  \end{equation}
\end{enumerate}

\end{defi}
Recall that the desuspension operator $s^{-1}$ is defined by mapping a graded vector space to a copy of itself shifted down  by $1$, i.e. $(s^{-1}\g)^i:=\g^{i+1}$.

\begin{lem}
Let $(\g,[-,-])$ be a graded Lie algebra. Then $(s^{-1}\g,\{-,-\})$ is a Lie $2$-algebra, where $\{s^{-1}x,s^{-1}y\}=(-1)^{|x|}s^{-1}[x,y]$ for homogenous elements $x,y\in\g$.	
\end{lem}

\begin{defi}
\begin{enumerate}[\rm(i)]
\item A Maurer-Cartan element of an $L_{\infty}$-algebra $(\frak g,\{l_i\}^{+\infty}_{i=1})$ is an element $\alpha\in \frak g^{0}$ satisfying the Maurer-Cartan equation \begin{equation}\label{9}\sum\limits^{+\infty}_{n=1}\frac{1}{n!}l_n(\alpha,\cdots,\alpha)=0.\end{equation}
\item A Maurer-Cartan element of an Lie $n$-algebra $(\frak g,\{-,\cdots,-\})$ is an element $\alpha\in \frak g^{0}$ satisfying the Maurer-Cartan equation \begin{equation}\label{10}
    \frac{1}{n!}\{\alpha,\cdots,\alpha\}_\g=0.
    \end{equation}

\end{enumerate}
\end{defi}

Let $\alpha$ be a Maurer-Cartan element of a Lie $n$-algebra $(\frak g,\{-,\cdots,-\})$. For all $k\geq1$ and $x_1,\cdots,x_n\in\frak g$. Define $l^{\alpha}_k:\otimes^k\frak g\rightarrow \frak g$ by
\begin{eqnarray}
  l^{\alpha}_k(x_1,\cdots,x_k) &=&  \frac{1}{(n-k)!}\{\underbrace{\alpha,\cdots,\alpha}_{n-k},x_1,\cdots,x_k\}_{\frak g},  \quad \forall~k\leq n,  \\
  l_k &=& 0, \quad \forall~k\geq n+1.
\end{eqnarray}

\begin{thm}{\rm(\cite{Get09})}
With the above notation, $(\frak g,l^{\alpha}_1,\cdots,l^{\alpha}_n)$ is an $L_{\infty}$-algebra, obtained from the Lie $n$-algebra $(\frak g,\{-,\cdots,-\})$ by twisting with the Maurer-Cartan element $\alpha$. Moreover, $\alpha+\alpha'$ is a Maurer-Cartan element of $(\frak g,\{-,\cdots,-\})$ if and only if $\alpha'$ is a Maurer-Cartan element of the twisted $L_{\infty}$-algebra $(\frak g,l^{\alpha}_1,\cdots,l^{\alpha}_n)$.
\end{thm}

\begin{defi}{\rm(\cite{Vor05})}
A {\bf $V$-data} is a quadruple $((L,[-,-]),\mathfrak h,{\rm P},\Delta)$, where $(L,[-,-])$ is a graded Lie algebra, $\mathfrak h$ is an abelian graded Lie subalgebra of $(L,[-,-])$, ${\rm P}: L\rightarrow L$ is a projection whose image is $\mathfrak h$ and kernel is a graded Lie subalgebra of $(L,[-,-])$ and $\Delta$ is an element in $\Ker~({\rm P})^1$ satisfying  $[\Delta,\Delta]=0$.
\end{defi}

\begin{thm}\label{th:V-data-L-infty}{\rm(\cite{Vor05})}{
Let $((L,[-,-]),\mathfrak h,{\rm P},\Delta)$ be a V-data. Then $(\mathfrak h,\{l_k\}^{+\infty}_{k=1})$ is an $L_{\infty}$-algebra where
\begin{equation}\label{12}
  l_k(a_1,\cdots,a_k)={\rm P}[\cdots[[\Delta,a_1],a_2],\cdots,a_k]\end{equation}
  for homogeneous $a_1,\cdots,a_k\in \mathfrak h$.
}\end{thm}
We call $\{l_k\}^{+\infty}_{k=1}$ the higher derived brackets of V-data $(L,\mathfrak h,{\rm P},\Delta)$.

\subsection{Maurer-Cartan characterization of relative Rota-Baxter operators on $n$-\LP pairs}
The notion of relative Rota-Baxter operators on $n$-\LP pairs are generalization of relative Rota-Baxter operators on both Lie algebras introduced in \cite{Kuper1} and $3$-Lie algebras introduced in \cite{BaiC19}.

\begin{defi}\label{rtbo:defi}
Let $(\frak g,[-,\cdots,-]_{\frak g};\rho)$ be an $n$-\LP pair.
A linear operator $T: V \rightarrow \frak g$ is called a {\bf relative Rota-Baxter operator} on an $n$-\LP pair $(\frak g,[-,\cdots,-]_{\frak g};\rho)$ if T satisfies
\begin{equation}\label{13}
  [Tv_1,\cdots,Tv_n]_{\frak g}=\sum\limits_{i=1}^n(-1)^{n-i}T(\rho(Tv_1,\cdots,\widehat{Tv}_i,\cdots,Tv_n)(v_i)),
\end{equation}
where $v_1,v_2,\cdots,v_n\in V$.
\end{defi}

\begin{rmk}
A {\bf Rota-Baxter operator} $T:\g\rightarrow\g$ on an $n$-Lie algebra $\g$ is a relative Rota-Baxter operator on an $n$-\LP pair $(\g;\ad)$. Furthermore, if  the $n$-Lie algebra reduces to a Lie algebra $(\g,[-,-])$, then the resulting linear operator $T:\g\rightarrow \g$ is a Rota-Baxter operator on the Lie algebra $\g$ introduced by the physicists C.-N. Yang and R. Baxter and if  the $n$-Lie algebra reduces to a $3$-Lie algebra $(\g,[-,-,-])$, then the resulting linear operator $T:\g\rightarrow \g$ is a Rota-Baxter operator on the $3$-Lie algebra $\g$ given in \cite{BR}.\end{rmk}

Consider the graded vector space
$$C^*(V,\frak g)=\oplus_{m\geq0}C^{m}(V,\frak g)=\Hom(\otimes^m(\wedge^{n-1}V)\wedge V,\frak g).$$

Define an $n$-linear operator $\{-,\cdots,-\}:C^{m_1}(V,\frak g)\times C^{m_2}(V,\frak g)\times\cdots \times C^{m_n}(V,\frak g) \rightarrow C^{m_1+\cdots+m_n+1}(V,\frak g)$ by
\begin{equation}\label{eq16}
 \{P_1,P_2,\cdots,P_n\}=[[[\mu+\rho,P_1]_{\nl},P_2]_{\nl},\cdots,P_n]_{\nl}.
\end{equation}

\begin{thm}\label{thm:Maurer-Cartan element}
With the above notations, $(C^*(V,\frak g),\{-,\cdots,-\})$ is a Lie $n$-algebra. Moreover, its Maurer-Cartan elements are precisely relative Rota-Baxter operators on the $n$-\LP pair $(\frak g,[-,\cdots,-]_{\frak g};\rho)$.
\end{thm}
\begin{proof}
Let $(V;\rho)$ be a representation of $n$-Lie algebra $(\frak g,\mu)$. Then the following quadruple gives a $V$-data:
\begin{itemize}
  \item the graded Lie algebra $(L,[-,-])$ is given by $(C^*_{\nl}(\frak g\oplus V,\frak g\oplus V),[-,-]_{\nl})$;
  \item the abelian graded Lie subalgebra $\mathfrak h$ is given by $\mathfrak h=C^*(V,\frak g)=\oplus_{m\geq0}\Hom(\otimes^m(\wedge^{n-1}V)\wedge V,\g)$;
  \item ${\rm P}:L\rightarrow L$ is the projection onto the subspace $\mathfrak h$;
  \item $\Delta=\mu+\rho\in \Ker~({\rm P})^1,~[\Delta,\Delta]_{\nl}=0$.
\end{itemize}
By Theorem \ref{th:V-data-L-infty}, $(\mathfrak h,\{l_k\}^{+\infty}_{k=1})$ is an $L_{\infty}$-algebra, where $l_k$ is given by Eq. (\ref{12}).
Note that $\|\mu\|=n-1|0$ and $\|\rho\|=n-1|0$ for $\mu\in \Hom(\wedge^n \g,\g)$ and $\rho\in \Hom(\wedge^{n-1} \g\otimes V,V)$.
For any $P_i\in \Hom(\otimes^{m_i}(\wedge^{n-1} V)\wedge V,\g)$, we have $\|P_i\|=-1|(n-1)m_i+1, 1\leq i\leq n$.
By Theorem \ref{thm:bidegree}, we have
\begin{gather*}
                        \|[\mu+\rho,P_1]_{\nl}\|=n-2|(n-1)m_1+1,  \\
                        \|[[\mu+\rho,P_1]_{\nl},P_2]_{\nl}\|=n-3|(n-1)(m_1+m_2)+2,\\
                       \vdots\\
                       \|[[[\mu+\rho,P_1]_{\nl},P_2]_{\nl},\cdots,P_{n-1}]_{\nl}\|=0|(n-1)(m_1+\cdots+m_{n-1}+1),\\
               \|[[[\mu+\rho,P_1]_{\nl},P_2]_{\nl},\cdots,P_n]_{\nl}\|=-1|(n-1)(m_1+\cdots+m_n+1)+1,
                       \end{gather*}
which imply that \begin{gather*}
               [\mu+\rho,P_1]_{\nl}\in \Ker~({\rm P}),\\
               [[\mu+\rho,P_1]_{\nl},P_2]_{\nl}\in \Ker~({\rm P}), \\
               \vdots\\
               [[[\mu+\rho,P_1]_{\nl},P_2]_{\nl},\cdots,P_{n-1}]_{\nl}\in\mathfrak \Ker~({\rm P}),\\
               [[[\mu+\rho,P_1]_{\nl},P_2]_{\nl},\cdots,P_n]_{\nl}\in\mathfrak h.
               \end{gather*}
Thus, we deduce $l_k=0$ for all $k\geq 1, k\neq n$. Therefore $(C^*(V,\g),\{-,\cdots,-\}=l_n)$ is a Lie $n$-algebra.

By a direct calculation, we have
\begin{eqnarray*}
&&\{T,\cdots,T\}(v_1,\cdots,v_n)\\
&=&(n-1)!\sum\limits_{1\leq i_1<\cdots<i_{n-1}\leq n}[\mu+\rho,T]_{\nl}(\cdots,Tv_{i_1},\cdots,Tv_{i_{n-1}},\cdots)\\
&&-(n-1)!T\sum\limits_{1\leq i_1<\cdots<i_{n-2}\leq n}[\mu+\rho,T]_{\nl}(\cdots,Tv_{i_1},\cdots\,Tv_{i_{n-2}},\cdots)\\
&=&n![Tv_1,\cdots,Tv_n]_{\frkg}-[(n-1)!+(n-1)!C^{n-2}_{n-1}]T\sum\limits_{1\leq i_1<\cdots<i_{n-1}\leq n}(\mu+\rho)(\cdots,Tv_{i_1},\cdots,Tv_{i_{n-1}},\cdots)\\
&=&n![Tv_1,\cdots,Tv_n]_{\frkg}-n!T\sum\limits_{1\leq i_1<\cdots<i_{n-1}\leq n}(\mu+\rho)(\cdots,Tv_{i_1},\cdots,Tv_{i_{n-1}},\cdots)\\
&=&n![Tv_1,\cdots,Tv_n]_{\frkg}-n!\sum\limits^{n}_{i=1}(-1)^{n-i}T\rho(Tv_1,\cdots,\widehat{v_i},\cdots,Tv_n)(v_i)=0,
\end{eqnarray*}
which implies that the linear map $T\in \Hom(V,\g)$ is a Maurer-Cartan element of the Lie $n$-algebra $(C^*(V,\g),\{-,\cdots,-\})$ if and only if $T$ is a relative Rota-Baxter operator on the $n$-\LP pair $(\frak g,[-,\cdots,-]_{\frak g};\rho)$.
\end{proof}

\begin{rmk}
Consider the graded vector space $C^*(V,\frak g)=\Hom((\otimes^m V)\wedge V,\frak g).$ By Theorem \ref{thm:Maurer-Cartan element}, $(C^*(V,\frak g),\{-,-\})$ is a Lie $2$-algebra and its Maurer-Cartan elements are precisely relative Rota-Baxter operators on the \LP pair $(\frak g,[-,-]_{\frak g};\rho)$. In \cite{TBGS}, the authors showed that the graded vector $\tilde{C}^*(V,\frak g)=\Hom(\wedge^{m+1} V,\frak g)$  equipped with the bracket $\{-,-\}$ given by
$$\{P,Q\}=[[\mu+\rho,P]_{\rm NR},Q]_{\rm NR},\quad \forall~P\in \Hom(\wedge^{m_1+1},\g),Q\in\Hom(\wedge^{m_2+1},\g)$$
  is also a  Lie $2$-algebra and its Maurer-Cartan elements are also precisely relative Rota-Baxter operators on the \LP pair $(\frak g,[-,-]_{\frak g};\rho)$, where the bracket $[-,-]_{\rm NR}$ is the Nijenhuis-Richardson bracket on $\tilde{C}^*(V,\frak g)$. Therefore, we give a new Lie $2$-algebra whose Maurer-Cartan elements are precisely relative Rota-Baxter operators on \LP pairs.
  \end{rmk}

Let $T$ be a relative Rota-Baxter operator on an $n$-\LP pair $(\frak g,[-,\cdots,-]_{\frak g};\rho)$. Since $T$ is a Maurer-Cartan element of the Lie $n$-algebra $(C^*(V,\frak g),\{-,\cdots,-\})$ given by Theorem \ref{thm:Maurer-Cartan element}, we have the twisted $L_{\infty}$-algebra structure on $C^*(V,\frak g)$ as follows:
\begin{eqnarray}\label{eq:l^T}
  l^{T}_k(P_1,\cdots,P_k) &=&  \frac{1}{(n-k)!}\{\underbrace{T,\cdots,T}_{n-k},P_1,\cdots,P_k\},  \quad \forall~k\leq n,  \\
  l^T_k &=& 0, \quad \forall~k\geq n+1.
\end{eqnarray}
\begin{thm}
Let $T:V\rightarrow \g$ be a relative Rota-Baxter operator on an $n$-\LP pair $(\frak g,[-,\cdots,-]_{\frak g};\rho)$. Then for a linear map $T':V\rightarrow \g$, $T+T'$ is a relative Rota-Baxter operator if and only if $T'$ is a Maurer-Cartan element of the twisted $L_{\infty}$-algebra $(C^*(V,\frak g),l^{T}_1,\cdots,l^{T}_n)$, that is, $T'$ satisfies the Maurer-Cartan equation:
$$l^{T}_1(T')+\frac{1}{2}l^{T}_2(T',T')+\frac{1}{3!}l^{T}_3(T',T',T')+\cdots+\frac{1}{n!}l^T_n(\underbrace{T',\cdots,T'}_{n})=0.$$
\end{thm}
\begin{proof}
By Theorem \ref{thm:Maurer-Cartan element}, $T+T'$ is a relative Rota-Baxter operator if and only if $$\frac{1}{n!}\{\underbrace{T+T',\cdots,T+T'\}}_{n}=0.$$
Applying $\underbrace{\{T,\cdots,T\}}_{n}=0$, the above condition is equivalent to
\begin{equation*}
\frac{1}{n!}(C_n^1\{\underbrace{T,\cdots,T}_{n-1},T'\}+\cdots
+C_n^k\{\underbrace{T,\cdots,T}_{n-k},\underbrace{T',\cdots,T'}_{k}\}+\cdots+C_n^n\{\underbrace{T',\cdots,T'}_{n}\})=0.
\end{equation*}
That is, $l^{T}_1(T')+\frac{1}{2}l^{T}_2(T',T')+\frac{1}{3!}l^{T}_3(T',T',T')+\cdots+\frac{1}{n!}l^{T}_n(\underbrace{T',\cdots,T'}_{n})=0$, which implies that $T'$ is a Maurer-Cartan element of the twisted $L_{\infty}$-algebra $(C^*(V,\frak g),l^{T}_1,\cdots,l^{T}_n)$.
\end{proof}

\subsection{$n$-pre-Lie algebras and relative Rota-Baxter operators on $n$-\LP pairs}
Now we give the notion of an $n$-pre-Lie algebra, which is a generalization of $3$-pre-Lie algebra introduced in \cite{BaiC19}.
\begin{defi}
Let $\g$ be a vector space with a multilinear map $\{-,\cdots,-\}:\wedge^{n-1}\g\otimes \g\rightarrow\g$. The pair $(\g,\{-,\cdots,-\})$ is called an {\bf $n$-pre-Lie algebra} if for $x_1,\cdots,x_n,y_1,\cdots,y_n\in\g$, the following identities hold:
\begin{eqnarray}\label{eq:n-pre1}
\{x_1,\cdots,x_{n-1},\{y_1,\cdots,y_{n-1},y_n\}\}&=&\sum\limits_{i=1}^{n-1}\{y_1,\cdots,y_{i-1},[x_1,\cdots,x_{n-1},y_i]_C,y_{i+1},\cdots,y_n\}\\
\nonumber&&+\{y_1,\cdots,y_{n-1},\{x_1,\cdots,x_{n-1},y_n\}\},\\
\label{eq:n-pre2}\{[y_1,\cdots,y_n]_C,x_1,\cdots,x_{n-2},x_{n-1}\}&=&\sum\limits_{i=1}^{n}(-1)^{n-i}\{y_1,\cdots,\hat{y_i},\cdots,y_n,\{y_i,x_1,\cdots,x_{n-2},x_{n-1}\}\},
\end{eqnarray}
where\begin{equation}\label{eq:npreC}
    [x_1,\cdots,x_n]_C=\sum\limits_{i=1}^{n}(-1)^{n-i}\{x_1,\cdots,\hat{x_i},\cdots,x_n,x_i\}.
\end{equation}
\end{defi}

Recall that a {\bf pre-Lie algebra} is a pair $(\g,\star)$, where $\g$ is a vector space and  $\star:\g\otimes \g\longrightarrow \g$ is a bilinear multiplication
satisfying
\begin{equation}\label{eq:pre-Lie algebra}
(x\star y)\star z-x\star(y\star z)=(y\star x)\star
z-y\star(x\star z),\quad \forall~x,y,z\in \g.
\end{equation}
For  a $2$-pre-Lie algebra $(\g,\{-,-\})$, we set $x\star y=\{x,y\}$ for $x,y\in\g$. It is obvious that  Eq. $\eqref{eq:n-pre1}$ is  equivalent to
$$x_1\star(y_1\star y_2)=(x_1\star y_1)\star y_2-(y_1\star x)\star y_2+y_1\star(x_1\star y_2)$$
and Eq. $\eqref{eq:n-pre2}$ is equivalent to
$$(y_1\star y_2)\star x_1-(y_2\star y_1)\star x_1=-y_2\star(y_1\star x_1)+y_1\star(y_2\star x_1).$$
Then we have $$\mbox{Eq. } \eqref{eq:n-pre1}\Leftrightarrow \mbox{Eq. }\eqref{eq:n-pre2} \Leftrightarrow \mbox{Eq. }\eqref{eq:pre-Lie algebra}.$$
Thus $2$-pre-Lie algebra is a pre-Lie algebra. See the survey
\cite{Pre-lie algebra in geometry} and the references therein for
more details on pre-Lie algebras.

\begin{pro}\label{prop6.2pre-lie}
Let $(\g,\{-,\cdots,-\})$ be an $n$-pre-Lie algebra. Then the induced $n$-bracket $[-,\cdots,-]_C$ given by Eq. $(\ref{eq:npreC})$ defines an $n$-Lie algebra.
\end{pro}

\begin{proof}
By the skew-symmetry of the first $n-1$ variables, the induced $n$-bracket $[-,\cdots,-]_C$ given by Eq. $(\ref{eq:npreC})$ is skew-symmetric.
For $x_1,\cdots,x_{n-1},y_1,\cdots,y_n\in \g$, by Eq. (\ref{eq:n-pre1}) and Eq. (\ref{eq:n-pre2}), we have
\begin{eqnarray*}
 &&[x_1,\cdots,x_{n-1},[y_1,\cdots,y_n]_C]_C-\sum\limits_{i=1}^{n}[y_1,\cdots,y_{i-1},[x_1,\cdots,x_{n-1},y_i]_C,y_{i+1},\cdots,y_n]_C\\
 &=&\{x_1,\cdots,x_{n-1},\sum\limits_{i=1}^{n}(-1)^{n-i}\{y_1,\cdots,\hat{y_i},\cdots,y_n,y_i\}\}+\sum\limits_{i=1}^{n-1}(-1)^{n-i}\{x_1,\cdots,\hat{x_i},
   \cdots,x_{n-1},[y_1,\cdots,y_n]_C,x_i\}\\
 &&-\sum\limits_{i=1}^{n}(-1)^{n-i}\{y_1,\cdots,y_{i-1},\hat{y_i},y_{i+1},\cdots,y_n,
    \sum\limits_{j=1}^{n-1}(-1)^{n-j}\{x_1,\cdots,\hat{x_j},\cdots,x_{n-1},y_i,x_j\}+\{x_1,\cdots,x_{n-1},y_i\}\}\\
    &&-\sum\limits_{i=1}^n\sum\limits_{k=1,k\neq i}^{n}(-1)^{n-k}\{y_1,\cdots,\hat{y_k},\cdots,y_{i-1},
    [x_1,\cdots,x_{n-1},y_i]_C,y_{i+1},\cdots,y_n,y_k\}\\
    &=&\sum\limits_{i=1}^{n}(-1)^{n-i}\{x_1,\cdots,x_{n-1},\{y_1,\cdots,\hat{y_i},\cdots,y_n,y_i\}\}
    -\sum\limits_{i=1}^{n}(-1)^{n-i}\{y_1,\cdots,\hat{y_{i}},\cdots,y_n,\{x_1,\cdots,x_{n-1},y_i\}\}\\
    &&-\sum\limits_{i=1}^{n}\sum\limits_{k=1,k\neq i}^{n}(-1)^{n-i}\{y_1,\cdots,\hat{y_i},\cdots,y_{k-1},[x_1,\cdots,x_{n-1},y_k]_C,y_{k+1},\cdots,y_n,y_i\}\\
    &&+\sum\limits_{j=1}^{n-1}(-1)^{n-j}\{x_1,\cdots,\hat{x_j},\cdots,x_{n-1},[y_1,\cdots,y_n]_C,x_j\}\\
    &&-\sum\limits_{i=1}^{n}\sum\limits_{j=1}^{n-1}
    (-1)^{i+j}\{y_1,\cdots,\hat{y_i},\cdots,y_n,\{x_1,\cdots,\hat{x_j},\cdots,x_{n-1},y_i,x_j\}\}\\
    &=&\sum\limits_{j=1}^{n-1}(-1)^{j}\{[y_1,\cdots,y_n]_C,x_1,\cdots,\hat{x_j},\cdots,x_{n-1},x_j\}\\
    &&-\sum\limits_{j=1}^{n-1}\sum\limits_{i=1}^{n}(-1)^{n-i+j}\{y_1,\cdots,\hat{y_i},\cdots,y_n,\{y_i,x_1,\cdots,\hat{x_j},\cdots,x_{n-1},x_j\}\}=0.
\end{eqnarray*}
Thus $(\g,[-,\cdots,-]_C)$ is an $n$-Lie algebra.
\end{proof}

\begin{defi}
Let $(\g,\{-,\cdots,-\})$ be an $n$-pre-Lie algebra. The $n$-Lie algebra $(\g,[-,\cdots,-]_C)$ is called the {\bf sub-adjacent $n$-Lie algebra} of  $(\g,\{-,\cdots,-\})$ and $(\g,\{-,\cdots,-\})$ is called a {\bf compatible $n$-pre-Lie algebra} of the $n$-Lie algebra $(\g,[-,\cdots,-]_C)$ .
\end{defi}
Let $(\g,\{-,\cdots,-\})$ be an $n$-pre-Lie algebra. Define a skew-symmetric multi-linear map $L:\wedge^{n-1}\g\rightarrow \gl(\g)$
by
\begin{equation}\label{40}
  L(x_1,\cdots,x_{n-1})(x_n)=\{x_1,\cdots,x_{n-1},x_n\}, \qquad\forall~ x_1,\cdots,x_n\in \g.
\end{equation}

\begin{pro}\label{pro:n-pre-Lie rep}
With the above notations, $(\g;L)$ is a representation of the $n$-Lie algebra $(\g,[-,\cdots,-]_C)$.
\end{pro}
\begin{proof}
By \eqref{eq:n-pre1}, we have
  \begin{eqnarray*}
  &&\{x_1,\cdots,x_{n-1},\{y_1,\cdots,y_n\}\}-\{y_1,\cdots,y_{n-1},\{x_1,\cdots,x_{n-1},y_n\}\}\\
  &=&\sum\limits_{i=1}^{n-1}\{y_1,\cdots,y_{i-1},[x_1,\cdots,x_{n-1},y_i]_C,y_{i+1},\cdots,y_n\},
\end{eqnarray*}
which implies that $[L(\frkX),L(\frkY)](y_n)=L(\frkX\circ\frkY)(y_n)$ holds.

By \eqref{eq:n-pre2}, we have
\begin{equation*}
  \{[y_1,\cdots,y_n]_C,x_1,\cdots,x_{n-2},x_{n-1}\}=\sum\limits_{i=1}^{n}(-1)^{n-i}\{y_1,\cdots,\hat{y_i},\cdots,y_n,\{y_i,x_1,\cdots,x_{n-2},x_{n-1}\}\},
\end{equation*}
which implies
$$L(x_1,\cdots,x_{n-2},[y_1,\cdots,y_n]_C)(x_{n-1})=\sum\limits_{i=1}^{n}(-1)^{n-i}L(y_1,\cdots,\hat{y_i},\cdots,y_n)L(x_1,\cdots,x_{n-2},y_i)(x_{n-1}).$$
Thus $(\g;L)$ is a representation of the $n$-Lie algebra $(\g,[-,\cdots,-]_C)$.
\end{proof}

By Proposition \ref{prop6.2pre-lie} and Proposition \ref{pro:n-pre-Lie rep}, we have
\begin{cor}
 Let $\g$ be a vector space with a linear map $\{-,\cdots,-\}:\wedge^{n-1}\g\otimes \g\rightarrow \g$. Then $(\g,\{-,\cdots,-\})$ is an $n$-pre-Lie algebra if and only if the bracket $[-,\cdots,-]_C$ defined by Eq. (\ref{eq:npreC}) is an $n$-Lie algebra structure on $\g$ and the left multiplication operation $L$ defined by Eq. (\ref{40}) gives a representation of this $n$-Lie algebra.
\end{cor}

\begin{pro}\label{pro6.5}
Let $T:V\rightarrow \g$ be a relative Rota-Baxter operator on an $n$-\LP pair $(\frak g,[-,\cdots,-]_{\frak g};\rho)$. Then there exists an $n$-pre-Lie algebra structure on $V$ given by
\begin{equation}
\{u_1,\cdots,u_n\}_T=\rho(Tu_1,\cdots,Tu_{n-1})(u_n), \quad\forall~ u_1,\cdots,u_n\in V.
\end{equation}
Furthermore, $(V,[-,\cdots,-]_T)$ is the sub-adjacent $n$-Lie algebra of the $n$-pre-Lie algebra $(\g,\{-,\cdots,-\}_T)$, where the bracket $[-,\cdots,-]_T:\wedge^n\g\rightarrow \g$ is given by
\begin{equation}\label{eqkj}
  [u_1,\cdots,u_n]_{T}=\sum\limits^{n}_{i=1}(-1)^{n-i}\rho(Tu_1,\cdots,\widehat{Tu_i},\cdots,Tu_n)(u_i)
\end{equation} and $T$ is an $n$-Lie algebra morphism from $(V,[-,\cdots,-]_T)$ to $(\g,[-,\cdots,-]_\g)$.
\end{pro}
\begin{proof}
It is obvious that  $\{-,\cdots,-\}_T\in\Hom(\wedge^{n-1}V\otimes V, V)$. Since $T:V\rightarrow \g$ is a relative Rota-Baxter operator and $(V;\rho)$ is a representation, by Eq. \eqref{eq:rep1} in Definition \ref{defi:rep},  we have
\begin{eqnarray*}
&&\{u_1,\cdots,u_{n-1},\{v_1,\cdots,v_n\}_T\}_T-\{v_1,\cdots,v_{n-1},\{u_1,\cdots,u_{n-1},v_n\}_T\}_T\\
&&-\sum\limits^{n-1}_{i=1}\{v_1,\cdots,v_{i-1},[u_1,\cdots,u_{n-1},v_i]_T,v_{i+1},\cdots,v_n\}_T\\
&=&\rho(Tu_1,\cdots,T{u_{n-1}})\rho(Tv_1,\cdots,Tv_{n-1})(v_n)-\rho(Tv_1,\cdots,Tv_{n-1})\rho(Tu_1,\cdots,T{u_{n-1}})(v_n)\\
&&-\sum\limits^{n-1}_{i=1}\rho(Tv_1,\cdots,Tv_{i-1},[Tu_1,\cdots,Tu_{n-1},Tv_i]_\g,Tv_{i+1},\cdots,Tv_{n-1})(v_n)=0.
\end{eqnarray*}
Similarly, by Eq. \eqref{eq:rep2} in Definition \ref{defi:rep}, we have
\begin{eqnarray*}
&&\{[v_1,\cdots,v_n]_T,u_1,\cdots,u_{n-1}\}_T-\sum\limits_{i=1}^{n}(-1)^{n-i}\{v_1,\cdots,\hat{v_i},\cdots,v_n,\{v_i,u_1,\cdots,u_{n-2},u_{n-1}\}_T\}_T\\
&=&\rho(T[v_1,\cdots,v_n]_T,Tu_1,\cdots,Tu_{n-2})(u_{n-1})-
\sum\limits_{i=1}^{n}(-1)^{n-i}\rho(Tv_1,\cdots,\widehat{Tv_i},\cdots,Tv_n)\rho(Tv_i,Tu_1,\cdots,Tu_{n-2})(u_{n-1})\\
&=&(-1)^{n-2}\big(\rho(Tu_1,\cdots,Tu_{n-2},[Tv_1,\cdots,Tv_n]_T)(u_{n-1})\\
&&-\sum\limits_{i=1}^{n}(-1)^{n-i}\rho(Tv_1,\cdots,\widehat{Tv_i},\cdots,Tv_n)\rho(Tu_1,\cdots,Tu_{n-2},Tv_i)(u_{n-1})\big)=0.
\end{eqnarray*}
The rest follows immediately.
\end{proof}

\begin{pro}\label{proB3.27}
Let  $(\g,[-,\cdots,-]_\g)$ be an $n$-Lie algebra. Then there is a compatible $n$-pre-Lie algebra if and only if there exists an invertible relative Rota-Baxter operator $T:V\rightarrow\g$ on an $n$-\LP pair $(\g;\rho)$. Moreover, the compatible $n$-pre-Lie structure on $\g$ is given by
\begin{equation}
 \{x_1,\cdots,x_n\}_\g=T\rho(x_1,\cdots,x_{n-1})(T^{-1}x_n),\quad\forall~x_1,\cdots,x_n\in\g.
\end{equation}
\end{pro}
\begin{proof}
Let $T$ be an invertible relative Rota-Baxter operator on an $n$-\LP pair $(\g;\rho)$. Then there exists an $n$-pre-Lie algebra structure on $V$ defined by $$\{u_1,\cdots,u_n\}_T=\rho(Tu_1,\cdots,Tu_{n-1})(u_n), ~\forall~ u_1,\cdots,u_n\in V.$$
Moreover, there is an induced $n$-pre-Lie algebra structure $\{-,\cdots,-\}_\g$ on $\g=T(V)$ given by
$$\{x_1,\cdots,x_n\}_\g=T\{T^{-1}x_1,\cdots,T^{-1}x_n\}_T=T\rho(x_1,\cdots,x_{n-1})(T^{-1}x_n)$$
for all $x_1,\cdots,x_n\in \g$. Since $T$ is a relative Rota-Baxter operator, we have $$[x_1,\cdots,x_n]_\g=T(\sum\limits_{i=1}^{n}(-1)^{n-i}\rho(x_1,\cdots,\hat{x_i},\cdots,x_n)(T^{-1}x_i))=
\sum\limits_{i=1}^{n}(-1)^{n-i}\{x_1,\cdots,\hat{x_i},\cdots,x_n,x_i\}_\g.$$
Therefore $(\g,\{-,\cdots,-\}_\g)$ is a compatible $n$-pre-Lie algebra of $n$-Lie algebra $(\g,[-,\cdots,-]_\g)$.

Conversely, the identity map $\id:\g\rightarrow\g$ is an invertible relative Rota-Baxter operator on the $n$-\LP pair $(\g;L)$.
\end{proof}

\begin{defi}
A {\bf symplectic structure}  on  an $n$-Lie algebra $(\g,[-,\cdots,-]_\g)$ is a nondegenerate skew-symmetric bilinear form $\omega\in\wedge^2\g^*$ satisfies the following condition:
\begin{equation}\label{eq:relationB}
\omega([x_1,\cdots,x_n]_\g,y)=-\sum\limits_{i=1}^{n}(-1)^{n-i}\omega(x_i,[x_1,\cdots,\hat{x_i},\cdots,x_n,y]_\g),\quad\forall~x_1,\cdots,x_n,y\in A.
\end{equation}
\end{defi}

\begin{pro}
Let $(\g,[-,\cdots,-]_\g,\omega)$ be a symplectic $n$-Lie algebra. Then there exists a compatible $n$-pre-Lie algebra structure $\{-,\cdots,-\}_\g$ on $\g$ defined by \begin{equation}
\omega(\{x_1,\cdots,x_n\}_\g,y)=-\omega(x_n,[x_1,\cdots,x_{n-1},y]_\g),\quad\forall~x_1,\cdots,x_n,y\in A.
\end{equation}
\end{pro}
\begin{proof}
Define a linear map $T: \g^*\rightarrow A$ by $\langle T^{-1}x,y\rangle=\omega(x,y)$ for all $x,y\in \g$. It is straightforward to check  that $\omega$ is a symplectic structure on the $n$-Lie algebra $\g$ if and only if $T$ is an invertible relative Rota-Baxter operator on the $n$-\LP pair $(\g;\ad^*)$. Then by Proposition \ref{proB3.27}, there exists a compatible $n$-pre-Lie algebra on $\g$ given by $\{x_1,\cdots,x_n\}_\g=T(\ad_{x_1,\cdots,x_{n-1}}^*T^{-1}x_n)$ for $x_1,\cdots,x_n\in \g$.
Thus we have
\begin{eqnarray*}
  &&\omega(\{x_1,\cdots,x_n\}_\g,y)=\omega(T(\ad_{x_1,\cdots,x_{n-1}}^*T^{-1}x_n),y)\\
  &=&\langle \ad_{x_1,\cdots,x_{n-1}}^*T^{-1}x_n,y\rangle=-\langle T^{-1}x_n,\ad_{x_1,\cdots,x_{n-1}}y\rangle=-\omega(x_n,[x_1,\cdots,x_{n-1},y]_\g).
\end{eqnarray*}
This completes the proof.
\end{proof}

\section{Cohomology of relative Rota-Baxter operators on $n$-\LP pairs}\label{sec:cohomology}
Let $T$ be a relative Rota-Baxter operator on an $n$-\LP pair $(\frak g,[-,\cdots,-]_{\frak g};\rho)$. By Proposition \ref{pro6.5}, $(V,[-,\cdots,-]_T)$ is an $n$-Lie algebra, where the bracket $[-,\cdots,-]_T$ is given by \eqref{eqkj}. Furthermore, we have
\begin{lem}
Let $T$ be a relative Rota-Baxter operator on an $n$-\LP pair $(\frak g,[-,\cdots,-]_{\frak g};\rho)$.
Define $\rho_T: \wedge^{n-1}V \rightarrow \gl(\g)$ by
\begin{equation}\label{eqlem4.2}
 \rho_T(u_1,\cdots,u_{n-1})x=[Tu_1,\cdots,Tu_{n-1},x]_\g-\sum\limits^{n-1}_{i=1}(-1)^{n-i}T\rho(Tu_1,\cdots,\widehat{Tu_i},\cdots,Tu_{n-1},x)(u_i),
 \end{equation}
 where $u_1,\dots,u_{n-1}\in V$ and $x\in\g$. Then $(\g;\rho_{T})$ is a representation of the $n$-Lie algebra $(V,[-,\cdots,-]_{T})$.
\end{lem}
\begin{proof}
Let $T$ be a relative Rota-Baxter operator on an $n$-\LP pair $(\frak g,[-,\cdots,-]_{\frak g};\rho)$.
Define $\bar{T}: \g\oplus V\rightarrow \g\oplus V$ by
$$\bar{T}(x+v)=Tv, \qquad \forall x\in \g, v\in V.$$
It is obvious that $\bar{T}\circ\bar{T}=0$. Furthermore, $\bar{T}$ is a Nijenhuis operator on the semidirect product $n$-Lie algebra $\g\ltimes_{\rho}V$ (\cite{LiuJ16}).
\emptycomment{$[Nx_1,\cdots,Nx_n]_{\ltimes_{\rho}}=N(\sum\limits_{i_1<\cdots<i_{n-1}}[\cdots,Nx_{i_1},\cdots,Nx_{i_{n-1}},\cdots]_{\ltimes_{\rho}})$
\begin{eqnarray*}
  [Tv_1,\cdots,Tv_n]_{\g} &=& [\bar{T}(x_1+v_1),\cdots,\bar{T}(x_n+v_n)]_{\ltimes_{\rho}}\\
  &=&\bar{T}\sum\limits_{i_1<\cdots<i_{n-1}}[\cdots,\bar{T}(x_{i_1}+v_{i_1}),\cdots,\bar{T}(x_{i_{n-1}}+v_{i_{n-1}}),\cdots]_{\ltimes_{\rho}}\\
  &=&\bar{T}\sum\limits^{n}_{j=1}[Tv_1,\cdots,x_j+v_j,\cdots,Tv_n]_{\ltimes_{\rho}}\\
  &=&\bar{T}\sum\limits^{n}_{j=1}(-1)^{n-j}\rho(Tv_1,\cdots,\widehat{Tv_j},\cdots,Tv_n)v_j\\
  &=&\sum\limits^{n}_{j=1}(-1)^{n-j}T\rho(Tv_1,\cdots,\widehat{Tv_j},\cdots,Tv_n)v_j,
\end{eqnarray*}
where $j\in n/ \{i_1,i_2,\cdots,i_{n-1}\}$, which is satisfied with Definition \ref{rtbo:defi}.
Thus $(\g\oplus V,[-,\cdots,-]_{\bar{T}})$ is an $n$-Lie algebra, where the bracket $[-,\cdots,-]_{\bar{T}})$ is given
by $$[x_1,\cdots,x_n]_{\bar{T}}=\sum\limits_{i_1<\cdots<i_{n-1}}[\cdots,\bar{T}x_{i_1},\cdots,\bar{T}x_{i_{n-1}},\cdots]_{\ltimes_{\rho}}-\bar{T}
\sum\limits_{i_1<\cdots<i_{n-2}}[\cdots,\bar{T}x_{i_1},\cdots,\bar{T}x_{i_{n-2}},\cdots]_{\ltimes_{\rho}}.$$}
It was shown in \cite{LiuJ16} that if $N$ is a Nijenhuis operator on an $n$-Lie algebra $(\g,[-,\cdots,-]_{\frkg})$, $(\g,[-,\cdots,-]^{n-1}_{N})$ is an $n$-Lie algebra, where
\begin{eqnarray*}
\label{25}\{-,\cdots,-\}^{j}_{N}&=&\sum\limits_{i_1<\cdots<i_{j}}[\cdots,Nx_{i_1},\cdots,Nx_{i_{j}},\cdots]_\g-N[x_1,\cdots,x_n]^{j-1}_{N},\\
\label{26}\{x_1,\cdots,x_n\}^1_{N}&=&\sum\limits^{n}_{i=1}[x_1,\cdots,Nx_i,\cdots,x_n]_\g-N[x_1,\cdots,x_n]_\g,\quad 2\leq j\leq n-1.
\end{eqnarray*}
By a direct calculation, we have
\begin{eqnarray*}
  &&[x_1+u_1,\cdots,x_n+u_n]_{\bar{T}} \\
  &=& \sum\limits_{i_1<\cdots<i_{n-1}}[\cdots,\bar{T}(x_{i_1}+u_{i_1}),\cdots,\bar{T}(x_{i_{n-1}}+u_{i_{n-1}}),\cdots]_{\rho}\\
  &&-\bar{T}
\sum\limits_{i_1<\cdots<i_{n-2}}[\cdots,\bar{T}(x_{i_1}+u_{i_1}),\cdots,\bar{T}(x_{i_{n-2}}+u_{i_{n-2}}),\cdots]_{\rho}\\
  &=& \sum\limits^{n}_{i=1}[Tu_1,\cdots,x_i+u_i,\cdots,Tu_n]_{\rho}-\sum\limits^{n}_{i=1}\sum\limits_{j> i}(-1)^{n-j}T\rho(Tu_1,\cdots,x_i,\cdots,
  \widehat{Tu_j},\cdots,Tu_n)(u_j)\\
  &&-\sum\limits^{n}_{i=1}\sum\limits_{j<i}(-1)^{n-j}T\rho(Tu_1,\cdots,\widehat{Tu_j},\cdots,x_i,
  \cdots,Tu_n)(u_j)\\
  &=&\sum\limits^{n}_{i=1}(-1)^{n-i}[Tu_1,\cdots,\widehat{Tu_i},\cdots,Tu_n,x_i]_\g+\sum\limits^{n}_{i=1}(-1)^{n-i}\rho(Tu_1,\cdots,\widehat{Tu_i},\cdots,Tu_n)(u_i)\\
  &&-\sum\limits^{n}_{i=1}\sum\limits_{j>i}(-1)^{n-j}(-1)^{n-i-1}T\rho(Tu_1,\cdots,\hat{x_i},\cdots,
  \widehat{Tu_j},\cdots,Tu_n,x_i)(u_j)\\
  &&-\sum\limits^{n}_{i=1}\sum\limits_{j<i}(-1)^{n-j}(-1)^{n-i}T\rho(Tu_1,\cdots,\widehat{Tu_j},\cdots,\hat{x_i},\cdots,
  Tu_n,x_i)(u_j)\\
  &=&\sum\limits^{n}_{i=1}(-1)^{n-i}\big([Tu_1,\cdots,\widehat{Tu_i},\cdots,Tu_n,x_i]_\g-\sum\limits^{n}_{i=1}\sum\limits_{j>i}(-1)^{n-j-1}T\rho(Tu_1,\cdots,
  \widehat{Tu_i},\cdots,\widehat{Tu_j},\cdots,Tu_n,x_i)(u_j)\\
  &&-\sum\limits^{n}_{i=1}\sum\limits_{j<i}(-1)^{n-j}T\rho(Tu_1,\cdots,\widehat{Tu_j},\cdots,\widehat{Tu_i},\cdots,
  Tu_n,x_i)(u_j)\big)+\sum\limits^{n}_{i=1}(-1)^{n-i}\rho(Tu_1,\cdots,\widehat{Tu_i},\cdots,Tu_n)(u_i)\\
  &=&[u_1,\cdots,u_n]_{T}+\sum\limits^{n}_{i=1}(-1)^{n-i}\rho_{T}(u_1,\cdots,\widehat{u_i},\cdots,u_n)(x_i),
  \end{eqnarray*}
which implies that $(\g;\rho_{T})$ is a representation of the $n$-Lie algebra $(V,[-,\cdots,-]_{T})$.
\end{proof}
Let $\partial_T:C^{m}_{\nl}(V;\g)\rightarrow C^{m+1}_{\nl}(V;\g)~(m\geq1)$ be the corresponding coboundary operator of the $n$-Lie algebra $(V,[-,\cdots,-]_{T})$ with coefficients in the representation $(\g;\rho_{T})$. More precisely, $\partial_T:C^{m}_{\nl}(V;\g)\rightarrow C^{m+1}_{\nl}(V;\g) ~(m\geq 1)$ is given by
\begin{eqnarray*}
\nonumber&&(\partial_T f)(\mathfrak{U}_1,\cdots,\mathfrak{U}_m,u_{m+1})\\
&=&\sum\limits_{1\leq j< k\leq m}(-1)^jf(\mathfrak{U}_1,\cdots,\widehat{\mathfrak{U}_j},\cdots,\mathfrak{U}_{k-1},\mathfrak{U}_j\circ \mathfrak{U}_k,\mathfrak{U}_{k+1},\cdots,\mathfrak{U}_{m},u_{m+1})\\
&&+\sum\limits_{j=1}^m(-1)^{j}f(\mathfrak{U}_1,\cdots,\widehat{\mathfrak{U}_j},\cdots,\mathfrak{U}_{m},
[\mathfrak U_{j},u_{m+1}]_{T})\\
&&+\sum\limits_{j=1}^m(-1)^{j+1}\rho_T(\mathfrak U_{j})f(\mathfrak{U}_1,\cdots,\widehat{\mathfrak{U}_j},\cdots,\mathfrak{U}_{m},
u_{m+1})\\
&&+\sum\limits_{i=1}^{n-1}(-1)^{n+m-i+1}\rho_T(u^1_m,\cdots,\hat{u^i_m},\cdots,u^{n-1}_m,u_{m+1})f(\mathfrak{U}_1,\cdots,\mathfrak{U}_{m-1},
u^i_{m})
\end{eqnarray*}
for any $\mathfrak{U}_i=u_i^1\wedge\cdots\wedge u_i^{n-1}\in \wedge^{n-1}V, ~i=1,2,\cdots,m, u_{m+1}\in V.$

For any $\frkX\in \wedge^{n-1}\g$, we define $\delta_T(\frkX):V\rightarrow \g$ by
\begin{equation}\label{eq:0-cocycle}
\delta_T(\frkX)v=T\rho(\frkX)v-[\frkX,Tv]_{\g}, ~v\in V.
\end{equation}

\begin{pro}
Let $T$ be a relative Rota-Baxter operator on an $n$-\LP pair $(\frak g,[-,\cdots,-]_{\frak g};\rho)$. Then $\delta_T(\frkX)$ is a $1$-cocycle on the $n$-Lie algebra $(V,[-,\cdots,-]_T)$ with coefficients in $(\g;\rho_{T})$.
\end{pro}
\begin{proof}
For any $u_1,\cdots,u_n\in V$, by the fact that $T$ is a relative Rota-Baxter operator, we have
\begin{eqnarray*}
   &&(\partial_T\delta_T(\frkX))(u_1,\cdots,u_n)\\
   &=&\sum\limits_{i=1}^n(-1)^{n-i}[Tu_1,\cdots,\widehat{Tu_i},\cdots,Tu_n,T\rho(\frkX)u_i-[\frkX,Tu_i]_{\g}]_{\g}\\
   &&-T\rho(\frkX)(\sum\limits_{i=1}^n(-1)^{n-i}\rho(Tu_1,\cdots,\widehat{Tu_i},\cdots,Tu_n)(u_i))+[\frkX,T(\sum\limits_{i=1}^n(-1)^{n-i}\rho(Tu_1,\cdots,\widehat{Tu_i},\cdots,Tu_n)(u_i))]_\g\\
   &&-T(\sum\limits_{i=1}^n\sum\limits_{j=1,j\neq i}^n(-1)^{n-j}\rho(Tu_1,\cdots,\widehat{Tu_j},\cdots,Tu_{i-1},T\rho(\frkX)(u_i)-[\frkX,Tu_i]_\g,Tu_{i+1},\cdots,Tu_n)(u_j))\\
   &=&\sum\limits_{i=1}^n(-1)^{n-i}[Tu_1,\cdots,\widehat{Tu_i},\cdots,Tu_n,T\rho(\frkX)(u_i)]_{\g}
   -\sum\limits_{i=1}^n(-1)^{n-i}[Tu_1,\cdots,\widehat{Tu_i},\cdots,Tu_n,[\frkX,Tu_i]_{\g}]_\g\\
   &&-T\rho(\frkX)(\sum\limits_{i=1}^n(-1)^{n-i}\rho(Tu_1,\cdots,\widehat{Tu_i},\cdots,Tu_n)(u_i))+[\frkX,[Tu_1,\cdots,Tu_i,\cdots,Tu_n]_\g]_{\g}\\
   &&-T(\sum\limits_{i=1}^n\sum\limits_{j=1,j\neq i}^n(-1)^{n-j}\rho(Tu_1,\cdots,\widehat{Tu_j},\cdots,Tu_{i-1},T\rho(\frkX)(u_i)-[\frkX,Tu_i]_\g,Tu_{i+1},\cdots,Tu_n)(u_j))\\
   &{=}&\sum\limits_{i=1}^n[Tu_1,\cdots,T\rho(\frkX)u_i,\cdots,Tu_n]_{\g}
   -T\rho(\frkX)(\sum\limits_{i=1}^n(-1)^{n-i}\rho(Tu_1,\cdots,\widehat{Tu_i},\cdots,Tu_n)(u_i))\\
   &&-T(\sum\limits_{i=1}^n\sum\limits_{j=1,j\neq i}^n(-1)^{n-j}\rho(Tu_1,\cdots,\widehat{Tu_j},\cdots,Tu_{i-1},T\rho(\frkX)(u_i)-[\frkX,Tu_i]_\g,Tu_{i+1},\cdots,Tu_n)(u_j))\\
   &{=}&\sum\limits_{i=1}^n(-1)^{n-i}T\rho(Tu_1,\cdots,\widehat{Tu_i},\cdots,Tu_n)\rho(\frkX)(u_i)\\
   &&+\sum\limits_{i=1}^n\sum\limits_{j=1,j\neq i}
   (-1)^{n-j}T\rho(Tu_1,\cdots,\widehat{Tu_j},\cdots,T\rho(\frkX)u_i,\cdots,Tu_n)(u_j)\\
   &&-\sum\limits_{j=1}^n(-1)^{n-j}T\rho(\frkX)\rho(Tu_1,\cdots,\widehat{Tu_j},\cdots,Tu_n)(u_j)\\
   &&-\sum\limits_{i=1}^n\sum\limits_{j=1,j\neq i}
   (-1)^{n-j}T\rho(Tu_1,\cdots,\widehat{Tu_j},\cdots,T\rho(\frkX)u_i,\cdots,Tu_n)(u_j)\\
   &&+\sum\limits_{i=1}^n\sum\limits_{j=1,j\neq i}^n(-1)^{n-j}T\rho(Tu_1,\cdots,\widehat{Tu_j},\cdots,Tu_{i-1},[\frkX,Tu_i]_\g,Tu_{i+1},\cdots,Tu_n)(u_j)\\
   &=&\sum\limits_{i=1}^n(-1)^{n-i}T\rho(Tu_1,\cdots,\widehat{Tu_i},\cdots,Tu_n)\rho(\frkX)(u_i)\\
   &&-\sum\limits_{j=1}^n(-1)^{n-j}T\rho(\frkX)\rho(Tu_1,\cdots,\widehat{Tu_j},\cdots,Tu_n)(u_j)\\
   &&+\sum\limits_{i=1}^n\sum\limits_{j=1,j\neq i}^n(-1)^{n-j}T\rho(Tu_1,\cdots,\widehat{Tu_j},\cdots,Tu_{i-1},[\frkX,Tu_i]_\g,Tu_{i+1},\cdots,Tu_n)(u_j)\\
   &=&0,
\end{eqnarray*}
which implies that $\delta_T(\frkX)$ is a $1$-cocycle on the $n$-Lie algebra $(V,[-,\cdots,-]_T)$ with coefficients in $(\g;\rho_{T})$.
\end{proof}
\begin{defi}
Let $T$ be a relative Rota-Baxter operator on an $n$-\LP pair $(\frak g,[-,\cdots,-]_{\frak g};\rho)$.
Define the set of $m$-cochains by
\begin{eqnarray}
   C^m_T(V;\g)=
   \begin{cases}
     C^m_{\nl}(V;\g) & m\geq1,\\
    \wedge^{n-1}\g&m=0.
     \end{cases}
\end{eqnarray}
Define $d:C^m_T(V;\g)\rightarrow C_T^{m+1}(V;\g)$ by
 \begin{eqnarray}
   d=
   \begin{cases}
     \partial_T & m\geq1,\\
    \delta _T& m=0,
     \end{cases}
\end{eqnarray}
where $\delta_T$ is given by Eq. \eqref{eq:0-cocycle}.
\end{defi}

Denote the set of $m$-cocycles by $\huaZ^m(V;\g)$ and the set of $m$-coboundaries by $\huaB^m(V;\g)$. The $m$-th cohomology group for the relative Rota-Baxter operator $T$ is denoted by \begin{equation}\label{31}
  \huaH^m(V;\g)=\huaZ^m(V;\g)/\huaB^m(V;\g),~m\geq 0.
\end{equation}
The relation between the coboundary operator $d$ and the differential $l^T_1$ defined by Eq. (\ref{eq:l^T}) using the Maurer-Cartan element $T$ of the Lie $n$-algebra $(C^*(V,\g),\{-,\cdots,-\})$ is given by the following theorem.
\begin{thm}\label{th:relationdl}
Let $T$ be a relative Rota-Baxter operator on an $n$-\LP pair $(\frak g,[-,\cdots,-]_{\frak g};\rho)$.
Then we have $df=(-1)^{m-1}l^T_1f, ~\forall~ f \in C^m_T(V;\g)$.
\end{thm}
\begin{proof}
For all $x_1,\cdots,x_n\in\g,u_1,\cdots,u_n\in V$, by a direct calculation, we have
{\footnotesize\begin{eqnarray*}
&&[\cdots,[[\mu+\rho,\underbrace{T]_{\nl},T]_{\nl},\cdots,T}_{n-1}]_{\nl}(x_1+u_1,\cdots,x_n+u_n)\\
 &=&(n-1)!\sum\limits^n_{i=1}[Tu_1,Tu_2,\cdots,x_i,\cdots,Tu_n]_{\g}+(n-1)!\sum\limits^n_{i=1}(-1)^{n-i}\rho(Tu_1,Tu_2,\cdots,\widehat{u_i},\cdots,Tu_n)(u_i)\\
 &&-(n-1)!\sum\limits^n_{i\neq j}(-1)^{n-j}T\rho(Tu_1,\cdots,x_i,\cdots,\widehat{Tu_j},\cdots,Tu_n)(u_j).
 \end{eqnarray*}}
Thus, we have
{\footnotesize\begin{eqnarray}
 \label{eq:RB-coboundary1}&&{[\cdots,[[\mu+\rho,\underbrace{T]_{\nl},T]_{\nl},\cdots,T}_{n-1}]_{\nl}}(u_1,\cdots,x_i,\cdots,u_n)\\
\nonumber &=&(n-1)!(-1)^{n-i}\rho_T(u_1,\cdots,\widehat{u_i},\cdots,u_n)x_i,\quad  1\leq i\leq n;\\
 \label{eq:RB-coboundary2} &&{[\cdots,[[\mu+\rho,\underbrace{T]_{\nl},T]_{\nl},\cdots,T}_{n-1}]_{\nl}}(u_1,\cdots,u_n)\\
 \nonumber&=&(n-1)![u_1,\cdots,u_n]_T.
\end{eqnarray}}
Moreover, for all $\frkU_i=u^1_i\wedge \cdots \wedge u_i^{n-1}\in \wedge^{n-1}V,~i=1,2\cdots,m$ and $u\in V$, we have
{\footnotesize\begin{eqnarray*}
&&\{\underbrace{T,\cdots,T}_{n-1},f\}(\frkU_1,\cdots,\frkU_m,u)\\
&=&[[\cdots[[\mu+\rho,\underbrace{T]_{\nl},T]_{\nl},\cdots,T}_{n-1}]_{\nl},f]_{\nl}(\frkU_1,\cdots,\frkU_m,u)\\
&=&\sum\limits_{i=1}^{n-1}[\cdots,[[\mu+\rho,\underbrace{T]_{\nl},T]_{\nl},\cdots,T}_{n-1}]_{\nl}(u_m^1\wedge\cdots\wedge u_m^{i-1}\wedge f(\frkU_1,\cdots,\frkU_{m-1},u_m^i)\wedge \cdots\wedge u_m^{n-1},u)\\
&&+\sum\limits_{\sigma\in S(1,m-1)}(-1)^{m-1}(-1)^{\sigma}[\cdots,[[\mu+\rho,\underbrace{T]_{\nl},T]_{\nl},\cdots,T}_{n-1}]_{\nl}(\frkU_{\sigma(1)},f(\frkU_{\sigma(2)},\cdots,\frkU_{\sigma(m)},u))\\
&&-\sum\limits_{j=1}^{n-1}\sum\limits_{k=1}^{m-1}(-1)^{m+k}\sum\limits_{\sigma\in S(k-1,1)}
(-1)^{\sigma}f(\frkU_{\sigma(1)},\cdots,\frkU_{\sigma(k-1)},u^1_{k+1}\wedge\cdots\wedge u_{k+1}^{j-1}\wedge\\
&&[\cdots,[[\mu+\rho,\underbrace{T]_{\nl},T]_{\nl},\cdots,T}_{n-1}]_{\nl}(\frkU_{\sigma(k)}, u_{k+1}^{j})\wedge u_{k+1}^{j+1}\wedge\cdots u_{k+1}^{n-1},\frkU_{k+2},\cdots,\frkU_m,u)\\
&&-\sum\limits_{\sigma\in S(m-1,1)}(-1)^{\sigma}f(\frkU_{\sigma(1)},\cdots,\frkU_{\sigma(m-1)},[\cdots,[[\mu+\rho,\underbrace{T]_{\nl},T]_{\nl},\cdots,T}_{n-1}]_{\nl}(\frkU_{\sigma(m)},u))\\
&=&\sum\limits_{i=1}^{n-1}(-1)^{n-i}(n-1)!
\rho_T(u_m^1,\cdots,u_m^{i-1},\hat{u_m^i},u_m^{i+1},\cdots,u_m^{n-1},u)f(\frkU_1,\cdots,\frkU_{m-1},u_m^i)\\
&&+\sum\limits_{i=1}^{m}(-1)^{m+i}[\cdots,[[\mu+\rho,\underbrace{T]_{\nl},T]_{\nl},\cdots,T}_{n-1}]_{\nl}(\frkU_i,f(\frkU_1,\cdots,\widehat{\frkU_i},\cdots,\frkU_{m},u))\\
&&-\sum\limits_{j=1}^{n-1}\sum\limits_{i=1}^{k}\sum\limits_{k=1}^{m-1}(-1)^{m+i}f(\frkU_1,\cdots,\widehat{\frkU_{i}},\cdots,\frkU_k,
u_{k+1}^1\wedge\cdots\wedge u_{k+1}^{j-1}\\
&&\wedge[\cdots,[[\mu+\rho,\underbrace{T]_{\nl},T]_{\nl},\cdots,T}_{n-1}]_{\nl}(\frkU_{i},u_{k+1}^j)\wedge u_{k+1}^{j+1}\wedge\cdots\wedge u_{k+1}^{n-1},\frkU_{k+2},
\cdots,\frkU_{m},u)\\
&&-\sum\limits_{i=1}^{m}(-1)^{m+i}f(\frkU_1,\cdots,\hat{\frkU_i},\cdots,\frkU_{m},[\cdots,[[\mu+\rho,\underbrace{T]_{\nl},T]_{\nl},\cdots,T}_{n-1}]_{\nl}(\frkU_{i},u))\\
&=&(-1)^{m-1}(n-1)!(\sum\limits_{i=1}^{n-1}(-1)^{n+m-i+1}\rho_T(u^1_m,\cdots,\hat{u^i_m},\cdots,u^{n-1}_m,u)f(\mathfrak{U}_1,\cdots,\mathfrak{U}_{m-1},
u^i_{m})\\
&&+\sum\limits_{i=1}^m(-1)^{i+1}\rho_T(\mathfrak U_{i})f(\mathfrak{U}_1,\cdots,\hat{\mathfrak{U}_i},\cdots,\mathfrak{U}_{m},u)+
\sum\limits_{i=1}^{m}(-1)^if(\frkU_1,\cdots,\widehat{\frkU_i},\cdots,\frkU_{m},[\frkU_{i},u]_T))\\
&&+(n-1)!(-1)^{m-1}\sum\limits_{1\leq i\leq k\leq m-1}(-1)^if(\frkU_1,\cdots,\widehat{\frkU_i},\cdots,\frkU_{k},
\sum\limits_{j=1}^{n-1}u_{k+1}^1\wedge\cdots\wedge u_{k+1}^{j-1}\\
&&\wedge[\frkU_i,u_{k+1}^j]_T
\wedge\cdots\wedge u_{k+1}^{n-1},\frkU_{k+2},
\cdots,\frkU_{m},u)\\
&=&(-1)^{m-1}(n-1)!df.
\end{eqnarray*}}
Thus, we deduce that $df=(-1)^{m-1}l^T_1f$.
\end{proof}

\section{Deformations of relative Rota-Baxter operators on $n$-\LP pairs}\label{sec:deformation}
 Let $V[[t]]$ denote the vector space of formal power series in $t$ with coefficients in $V$. If in addition, $(\g,[-,\cdots,-]_{\g})$ is an $n$-Lie algebra over $\K$, then there is an $n$-Lie algebra structure over the ring $\K[[t]]$ on $\g[[t]]$ given by
\begin{eqnarray}\label{eq:33defor}
[\sum\limits_{j_1=0}^{+\infty}a_1^{j_1}t^{j_1},\cdots,\sum\limits_{j_n=0}^{+\infty}a_n^{j_n}t^{j_n}]_{\g}
=\sum\limits_{s=0}^{+\infty}\sum\limits_{j_1+\cdots+j_n=s}[a_1^{j_1},\cdots,a_n^{j_n}]_{\g}t^s,
\end{eqnarray}
where $a_1^{j_1},\cdots,a_n^{j_n}\in \g$.

For any representation $(V;\rho)$ of $(\g,[-,\cdots,-]_{\g})$, there is a natural representation of the $n$-Lie algebra $\g[[t]]$ on the $\K[[t]]$-module $V[[t]]$ given by
\begin{eqnarray}\label{eq:34defor}
\rho(\sum\limits_{j_1=0}^{+\infty}a_1^{j_1}t^{j_1},\cdots,\sum\limits_{j_{n-1}=0}^{+\infty}a_{n-1}^{j_{n-1}}t^{j_{n-1}})(\sum\limits_{k=0}^{+\infty}v_kt^k)
=\sum\limits_{s=0}^{+\infty}\sum\limits_{j_1+\cdots+j_{n-1}=s}\rho(a_1^{j_1},\cdots,a_{n-1}^{j_{n-1}})(v_k)t^s
\end{eqnarray}
where $a_1^{j_1},\cdots,a_{n-1}^{j_{n-1}}\in \g, v_k\in V$.
Let $T$ be a relative Rota-Baxter operator on an $n$-\LP pair $(\frak g,[-,\cdots,-]_{\frak g};\rho)$. Consider a power series
\begin{equation}\label{eq:35defor}
  T_t=\sum_{i=0}^{+\infty}\frkT_it^i,\quad\frkT_i\in\Hom_{\K}(V,\g),
\end{equation}
that is, $T_t\in \Hom_{\K}(V,\g[[t]])$. Furthermore,  $T_t$ can be extended to be a $\K[[t]]$-module map from $V[[t]]$ to $\g[[t]]$.

\begin{defi}
If $T_t=\sum_{i=0}^{+\infty}\frkT_it^i$ with $\frkT_0=T$ satisfies
\begin{equation}\label{Eq:def5.136}
  [T_t(u_1),\cdots,T_t(u_n)]_{\g}=T_t(\sum\limits_{i=1}^m(-1)^{n-i}\rho(T_t(u_1),\cdots,\widehat{T_t(u_i)},\cdots,T_t(u_n))(u_i)),
\end{equation}
we say that $T_t$ is a {\bf formal deformation of relative Rota-Baxter operator }$T$.
\end{defi}

\emptycomment{\begin{pro}
If $T_t=\sum_{i=0}^{+\infty}\frkT_it^i$ is a formal deformation of a relative Rota-Baxter operator $T$ on an $n$-\LP pair $(\frak g,[-,\cdots,-]_{\frak g};\rho)$, then $[-,\cdots,-]_{T_t}$ defined by
\begin{equation*}
[u_1,\cdots,u_n]_{T_t}=\sum_{k=0}^{+\infty}\sum\limits_{j_1+\cdots+j_{n-1}=k}^m(\sum\limits_{i=1}^n
(-1)^{n-i}\rho(\frkT_{j_1}u_1,\cdots,\hat{u_i},\frkT_{j_i}u_{i+1},\cdots,\frkT_{j_{n-1}}u_n)(u_i))t^k,~\forall~u_1,\cdots,u_n\in V
\end{equation*}
is a formal deformation of the  $n$-Lie algebra $(V,[-,\cdots,-]_{T})$.
\end{pro}}
Applying Eqs. (\ref{eq:33defor})-(\ref{eq:35defor}) to expand Eq. (\ref{Eq:def5.136}) and comparing coefficients of $t^s$,  Eq. (\ref{Eq:def5.136}) is equivalent to the following equations
\begin{equation}\label{eqpro5.237}
  \sum\limits_{\mbox{\tiny$\begin{array}{c}
i_1+\cdots i_n=s\\
i_1,\cdots,i_n\geq 0\end{array}$}}[\frkT_{i_1}u_1,\cdots,\frkT_{i_n}u_n]_{\g}=\sum\limits_{\mbox{\tiny$\begin{array}{c}
i_1+\cdots i_n=s\\
i_1,\cdots,i_n\geq 0\end{array}$}}\frkT_{i_1}(\sum\limits_{k=1}^n(-1)^{n-k}\rho(\frkT_{i_2}u_1,\cdots,\frkT_{i_k}u_{k-1},\hat{u_k},\frkT_{i_{k+1}}u_{k+1},\cdots,\frkT_{i_n}u_n)(u_k)),
\end{equation}
for all $s\geq 0$ and $u_1,\cdots,u_n\in V$.

\begin{defi}
Let $T$ be a relative Rota-Baxter operator on an $n$-\LP pair $(\frak g,[-,\cdots,-]_{\frak g};\rho)$. An {\bf order $m$ deformation} of the relative Rota-Baxter operator $T$  is a sequence of linear maps $\frkT_i\in \Hom_{\K}(V,\g)$ for $0\leq i\leq m$ with $\frkT_0=T$ such that $\K[t]/(t^{m+1})$-module map $T_t=\sum_{i=0}^{m}\frkT_it^i$ from $V[t]/(t^{m+1})$ to the $n$-Lie algebra $\g[t]/(t^{m+1})$ satisfies
\begin{equation}\label{40}
  [T_t(u_1),\cdots,T_t(u_n)]_{\g}=T_t(\sum\limits_{i=1}^n(-1)^{n-i}\rho(T_t(u_1),\cdots,\widehat{T_t(u_i)},\cdots,T_t(u_n))(u_i)),
\end{equation}
where $u_1,u_2,\cdots u_n\in V$.
\end{defi}

We call an order $1$ deformation of a relative Rota-Baxter operator $T$ on an $n$-\LP pair $(\frak g,[-,\cdots,-]_{\frak g};\rho)$ an {\bf infinitesimal deformation} and denote it by $(\frak g,\frkT_1)$.

By direct calculations, $(\g,\frkT_1)$ is an infinitesimal deformation of a relative Rota-Baxter operator $T$ on an $n$-\LP pair $(\frak g,[-,\cdots,-]_{\frak g};\rho)$  if and only if for all $u_1,u_2,\cdots,u_n\in \g$,
\begin{eqnarray*}
\label{eq:1-cocycle}&&[\frkT_1u_1,Tu_2,\cdots,Tu_n]_\g+\cdots+[Tu_1,\cdots,\frkT_1u_n]_\g \\
\nonumber&=& T(\sum\limits_{i,j=1,i\neq j}^n(-1)^{n-j}\rho(Tu_1,\cdots,\widehat{Tu_j},\cdots,\frkT_1 u_i,Tu_{i+1},\cdots,Tu_n)(u_j))\\
\nonumber&&+\frkT_1(\sum\limits_{i=1}^n(-1)^{n-i}\rho(Tu_1,\cdots,\widehat{Tu_i},\cdots,Tu_n)(u_i)),
\end{eqnarray*}
which implies that $\frkT_1$ is a 1-cocycle for the relative Rota-Baxter operator $T$, i.e. $d\frkT_1=0$.

Two infinitesimal deformations  $(\g,\frkT_1)$ and $(\g,\frkT'_1)$  of a
relative Rota-Baxter operator $T$ are said to be {\bf
equivalent} if there exists $\frkX\in \wedge^{n-1}\g$, such that for $\phi_t=\Id_{\g}+t\ad_{\frkX}$ and $\varphi_t=\Id_{V}+t\rho(\frkX)$, the following conditions hold:
\begin{enumerate}
  \item[\rm(i)] $[\phi_t(x_1),\cdots,\phi_t(x_{n})]_{\g}=\phi_t[x_1,\cdots,x_n]_{\g}$ modulo $t^2$ for all $x_1,\cdots,x_n$;
  \item [\rm(ii)] $\varphi_t\rho(x_1,\cdots,x_{n-1})(u)=\rho(\phi_t(x_1),\cdots,\phi_t(x_{n-1}))(\varphi_t(u)))$ modulo $t^2$ for all $x_1,\cdots,x_{n-1}\in \g,u\in V$;
  \item[\rm(iii)] $T_t\circ \varphi_t=\phi_t\circ{T'_t}$  modulo $t^2$, where $T_t=T+t \frkT_1$ and $T'_t=T+t \frkT'_1$.
\end{enumerate}

By Eq. \eqref{1} in the definition of $n$-Lie algebra and Eq. \eqref{eq:rep1} in the definition of representation of an $n$-Lie algebra, the conditions (i) and (ii) follow.

By the relation $T_t\circ \varphi_t=\phi_t\circ{T'_t}$, we have
$${\frkT'_1}(v)=\frkT_1(v)+T\rho(\frkX)(v)-[\frkX,Tv]_\g=\frkT_1(v)+(d\frkX)(v), ~v\in V,$$
which implies that ${\frkT'_1}-\frkT_1=d\frkX$. Thus we have
\begin{thm}
 There is a one-to-one correspondence between the space of equivalence classes of infinitesimal deformations of a relative Rota-Baxter operator $T$ and the first cohomology group $\huaH^1(V,\g)$.
\end{thm}

\begin{defi}
Let $T_t=\sum_{i=0}^{m}\frkT_it^i$ be an order $m$ deformation of a relative Rota-Baxter operator $T$ on an $n$-\LP pair $(\frak g,[-,\cdots,-]_{\frak g};\rho)$. $T_t$ is said to be {\bf extendable} if there exists a $1$-cochain $\frkT_{m+1}\in \Hom(V,\g)$ such that $\tilde{T_t}=T_t+\frkT_{m+1}t^{m+1}$ is an order $m+1$ deformation of the relative Rota-Baxter operator $T$.
\end{defi}

Let $T_t=\sum_{i=0}^{m}\frkT_it^i$ be an order $m$ deformation of a relative Rota-Baxter operator $T$ on an  $n$-\LP pair $(\frak g,[-,\cdots,-]_{\frak g};\rho)$. Define $\Theta_m\in C_{T}^2(V;\g)$ by
\begin{eqnarray}
\label{eq37}\Theta_m(u_1,u_2,\cdots,u_n)&=&\sum\limits_{\mbox{\tiny$\begin{array}{c}
i_1+\cdots i_n=m+1\\
0\leq i_1,\cdots,i_n\leq m\end{array}$}}\Big([\frkT_{i_1}u_1,\cdots,\frkT_{i_n}u_n]_{\g}\\
\nonumber&&-\frkT_{i_1}(\sum\limits_{k=1}^n(-1)^{n-k}\rho(\frkT_{i_2}u_1,
\cdots,\frkT_{i_k}u_{k-1},\frkT_{i_{k+1}}u_{k+1},\cdots,\frkT_{i_n}u_n)(u_k))\Big),
\end{eqnarray}
where $u_1,\cdots,u_n\in V$.

\begin{pro}
The $2$-cochain $\Theta_m$ is a $2$-cocycle, that is, $d\Theta_m=0$.
\end{pro}

\begin{proof}
Using Eq. \eqref{eq16}, we have
\begin{eqnarray}\label{eq:2-cochain1}
&&\{\frkT_{i_1},\cdots,\frkT_{i_n}\}(u_1,\cdots,u_n)=\sum\limits_{\sigma\in S_n}[\frkT_{i_{\sigma(1)}}(u_1),\cdots,\frkT_{i_{\sigma(n)}}(u_n)]_\g\\
\nonumber&& -\sum\limits_{\sigma\in S_n}\frkT_{i_{\sigma(1)}}\sum\limits_{k=1}^n(-1)^{n-k}\rho\big(\frkT_{i_{\sigma(2)}}(u_1),\cdots,\frkT_{i_{\sigma(k)}}(u_{k-1}),\frkT_{i_{\sigma(k+1)}}(u_{k+1}),\cdots,\frkT_{i_{\sigma(n)}}(u_n)\big) (u_k).
 \end{eqnarray}
By Eq. \eqref{eq37} and Eq. \eqref{eq:2-cochain1}, we deduce that
\begin{equation}\label{eq:2-cochain2}
\Theta_m=1/n!\sum\limits_{\mbox{\tiny$\begin{array}{c}
i_1+\cdots i_n=m+1\\
0\leq i_1,\cdots,i_n\leq m\end{array}$}}\{\frkT_{i_1},\cdots,\frkT_{i_n}\}=\sum\limits_{j=0}^{n-2}\frac{C_n^j}{n!}\sum\limits_{\mbox{\tiny$\begin{array}{c}
i_1+\cdots i_{n-j}=m+1\\
1\leq i_1,\cdots,i_{n-j}\leq m\end{array}$}}\{\underbrace{T,\cdots,T}_{j},\frkT_{i_1},\cdots,\frkT_{i_{n-j}}\}.
\end{equation}
Since $T_t$ is an order $m$ deformation of the relative Rota-Baxter operator $T$, by \eqref{eqpro5.237} and \eqref{eq:2-cochain1}, we have
\begin{eqnarray}
-\frac{1}{(n-1)!}\{\underbrace{T,\cdots,T}_{n-1},\frkT_s\}
\label{eq:(42)s+1}&=&\sum\limits_{j=0}^{n-2}\frac{C_n^j}{n!}\sum\limits_{\mbox{\tiny$\begin{array}{c}
i_1+\cdots i_{n-j}=s\\
1\leq i_1,\cdots,i_{n-j}\leq s-1\end{array}$}}\{\underbrace{T,\cdots,T}_{j},\frkT_{i_1},\cdots,\frkT_{i_{n-j}}\}.
\end{eqnarray}

By Theorem \ref{th:relationdl} and Eq. (\ref{eq:l^T}), we have
{\footnotesize\begin{eqnarray*}
d \Theta_m&=&-\frac{1}{(n-1)!}\{\underbrace{T,\cdots,T}_{n-1},\Theta_m\}\\
&\stackrel{\eqref{eq:2-cochain2}}=&-\frac{1}{(n-1)!}\sum\limits_{j=0}^{n-2}\frac{C_n^j}{n!}\sum\limits_{\mbox{\tiny$\begin{array}{c}
i_1+\cdots i_{n-j}=m+1\\
1\leq i_1,\cdots,i_{n-j}\leq m\end{array}$}}\{\{\frkT_{i_1},\cdots,\frkT_{i_{n-j}},\underbrace{T,\cdots,T}_{j}\},T,\cdots,T\}\\
&\stackrel{\eqref{eq:general-JI}}=& \frac{1}{(n-1)!}\sum\limits_{j=0}^{n-2}\frac{C_n^j}{n!}\sum\limits_{\mbox{\tiny$\begin{array}{c}
i_1+\cdots i_{n-j}=m+1\\
1\leq i_1,\cdots,i_{n-j}\leq m\end{array}$}}\sum\limits_{k=j+1}^{n-1}\frac{C_{n-1+j}^k\cdot C_{n-j}^{n-k}}{C_{n-1+j}^{j}}\{\{\frkT_{i_1},\cdots,\frkT_{i_{n-k}},\underbrace{T,\cdots,T}_{k}\},\frkT_{i_{n-k+1}},\cdots,\frkT_{i_{n-j}},T,\cdots,T\}\\
&=&\frac{1}{(n-1)!}\sum\limits_{0\leq p<k\leq n-2}^{n-2}\frac{C_n^{p}}{n!}\sum\limits_{\mbox{\tiny$\begin{array}{c}
i_1+\cdots i_{n-p}=m+1\\
1\leq i_1,\cdots,i_{n-p}\leq m\end{array}$}}\frac{C_{n-1+p}^k\cdot C_{n-p}^{n-k}}{C_{n-1+p}^{p}}\{\{\frkT_{i_1},\cdots,\frkT_{i_{n-k}},\underbrace{T,\cdots,T}_{k}\},\frkT_{i_{n-k+1}},\cdots,\frkT_{i_{n-p}},T,\cdots,T\}\\
&&+\frac{1}{(n-1)!}\sum\limits_{j=0}^{n-2}\frac{C_n^j}{n!}\sum\limits_{\mbox{\tiny$\begin{array}{c}
i_1+\cdots i_{n-j}=m+1\\
1\leq i_1,\cdots,i_{n-j}\leq m\end{array}$}}C_{n-j}^1\{\{\frkT_{i_1},T,\cdots,T\},\frkT_{i_2},\cdots,\frkT_{i_{n-j}},T,\cdots,T\}\\
&\stackrel{\eqref{eq:(42)s+1}}=&\frac{1}{(n-1)!}\sum\limits_{0\leq {p}<k\leq n-2}^{n-2}\frac{C_n^{p}}{n!}\sum\limits_{\mbox{\tiny$\begin{array}{c}
i_1+\cdots i_{n-p}=m+1\\
1\leq i_1,\cdots,i_{n-p}\leq m\end{array}$}}\frac{C_{n-1+p}^k\cdot C_{n-p}^{n-k}}{C_{n-1+p}^{p}}\{\{\frkT_{i_1},\cdots,\frkT_{i_{n-k}},\underbrace{T,\cdots,T}_{k}\},\frkT_{i_{n-k+1}},\cdots,\frkT_{i_{n-p}},T,\cdots,T\}\\
&&-\sum\limits_{k=0}^{n-2}\frac{C_n^{k}}{n!}\sum\limits_{j=0}^{n-2}\sum\limits_{\mbox{\tiny$\begin{array}{c}
i^{'}_1+\cdots i^{'}_{n-k}+i_2+\cdots+i_{n-j}=m+1\\
1\leq i^{'}_1,\cdots,i^{'}_{n-k},i_2,\cdots,i_{n-j}\leq m\end{array}$}}\frac{C_n^{j}}{n!}C_{n-j}^1\{\{\frkT_{i_1^{'}},\cdots,\frkT_{i_{n-k}^{'}},\underbrace{T,\cdots,T}_{k}\},\frkT_{i_2},\cdots,
\frkT_{i_{n-j}},T,\cdots,T\}\\
&\stackrel{\eqref{eq:general-JI}}=&-\sum\limits_{\mbox{\tiny$\begin{array}{c}
0< k+j\leq n-2\\
0\leq k\leq n-2\end{array}$}}\frac{C_n^{j}}{n!}\frac{C_n^{k}}{n!}C_{n-j}^1\sum\limits_{\mbox{\tiny$\begin{array}{c}
i^{'}_1+\cdots i^{'}_{n-k}+i_2+\cdots+i_{n-j}=m+1\\
1\leq i^{'}_1,\cdots,i^{'}_{n-k},i_2,\cdots,i_{n-j}\leq m\end{array}$}}\{\{\frkT_{i_1^{'}},\cdots,\frkT_{i_{n-k}^{'}},\underbrace{T,\cdots,T}_{k}\},\frkT_{i_2},\cdots,\frkT_{i_{n-j}},T,\cdots,T\}\\
&\stackrel{\eqref{eq:general-JI}}=&0,
\end{eqnarray*}}
which implies that the $2$-cochain {$\Theta_m$} is a cocycle.
\end{proof}

Moreover, we have
\begin{thm}
Let $T_t=\sum_{i=0}^m\frkT_it^i$ be an order $m$ deformation of a relative Rota-Baxter operator $T$ on an $n$-\LP pair $(\frak g,[-,\cdots,-]_{\frak g};\rho)$. Then $T_t$ is extendable if and only if the cohomology group ${[\Theta_m]}$ in $\huaH^2(V,\g)$ is trivial.
\end{thm}
\begin{proof}
Assume that an order $m$ deformation $T_t$ of the relative Rota-Baxter operator $T$ can be extended to an order $m+1$ deformation. By Eq. \eqref{eqpro5.237} with $s=m+1$, Eq. (\ref{eq:(42)s+1}) holds. Thus, we have $\Theta_m=-d\frkT_{m+1}$, which implies that the cohomology group ${[\Theta_m]}$ is trivial.

Conversely, if the cohomology group  ${[\Theta_m]}$ is trivial, then there exists a $1$-cochain $\frkT_{m+1} \in \Hom_{\K}(V,\g)$ such that ${\Theta_m}=d\frkT_{m+1}$.
Set $\tilde{T_t}=T_t+\frkT_{m+1}t^{m+1}$. Then $\tilde{T_t}$ satisfies Eq. (\ref{eq:(42)s+1}) for $0\leq s\leq m+1$. Thus $\tilde{T_t}$ is an order $m+1$ deformation, which implies that $T_t$ is extendable.
\end{proof}

\section{From cohomology groups of relative Rota-Baxter operators on $n$-Lie algebras to those on $(n+1)$-Lie algebras }\label{sec:construction}

Motivated by the construction of $3$-Lie algebras from the general linear Lie algebras with trace
forms in \cite{trace form}, the authors in \cite{BWLZ13}  provide a construction of $(n+1)$-Lie algebras from $n$-Lie algebras and certain linear functions.
\begin{lem}{{\rm(\cite{BWLZ13})}\label{n-n+1}
Let $(\g,[-,\cdots,-]_\g)$ be an $n$-Lie algebra and $\g^*$ the dual space of $\g$. Suppose $f\in \g^*$ satisfies $f([x_1,\cdots,x_n]_\g)=0$ for all $x_i\in \g$. Then $(\g,\{-,\cdots,-\})$ is an $(n+1)$-Lie algebra, where the bracket is given by
\begin{equation}
\{x_1,\cdots,x_{n+1}\}=\sum\limits_{i=1}^{n+1}(-1)^{i-1}f(x_i)[x_1,\cdots,\hat{x_i},\cdots,x_{n+1}]_\g,\quad\forall~x_i\in\g.
\end{equation}
The $(n+1)$-Lie algebra constructed as above is denoted by $\g_f$.
}\end{lem}

\begin{pro}\label{pro:new LieRep}
Let $(\g,[-,\cdots,-]_{\g};\rho)$ be an $n$-\LP pair. Define $\varrho:\wedge^{n}\g\rightarrow\gl(V)$ by
\begin{equation}
\varrho(x_1,\cdots,x_n)=\sum\limits_{i=1}^{n}(-1)^{i-1}f(x_i)\rho(x_1,\cdots,\hat{x_i},\cdots,x_n).
\end{equation}
Then $(\g_f;\varrho)$ is an $(n+1)$-\LP pair.
\end{pro}
\begin{proof}
Since $(\g,[-,\cdots,-]_{\g};\rho)$ is an $n$-\LP pair and $f([x_1,\cdots,x_n])=0$ for all $x_i\in \g$, by a direct calculation, we have
\begin{eqnarray*}
&&[\varrho(x_1,\cdots,x_n),\varrho(y_1,\cdots,y_n)]-\sum_{i=1}^n\varrho(y_1,\cdots,y_{i-1},\{x_1,\cdots,x_n,y_i\},y_{i+1},\cdots,y_n)\\
&=&[\sum\limits_{k=1}^{n}(-1)^{k-1}f(x_k)\rho(x_1,\cdots,\hat{x_k},\cdots,x_n),
\sum\limits_{j=1}^{n}(-1)^{j-1}f(y_j)\rho(y_1,\cdots,\hat{y_j},\cdots,y_n)]\\
&&-\sum\limits_{i,j,k=1}^{n}(-1)^{k+j}f(x_k)f(y_j)\rho(y_1,\cdots,y_{i-1},[x_1,\cdots,\hat{x_k},\cdots,x_n,y_i]_\g,\cdots,\hat{y_j},\cdots,y_n)\\
&&-\sum\limits_{i,j=1}^{n}(-1)^{n+j-1}f(y_i)f(y_j)\rho(y_1,\cdots,y_{i-1},[x_1,\cdots,x_n]_\g,y_{i+1},\cdots,\widehat{y_j},\cdots,y_n)\\
&=&\sum\limits_{k=1}^{n}\sum\limits_{j=1}^{n}(-1)^{k+j}f(x_k)f(y_j)\big([\rho(x_1,\cdots,\hat{x_k},\cdots,x_n),
\rho(y_1,\cdots,\hat{y_j},\cdots,y_n)]\\
&&-\sum_{i=1}^n\rho(y_1,\cdots,y_{i-1},[x_1,\cdots,\hat{x_k},\cdots,y_i]_\g,\cdots,\hat{y_j},\cdots,y_n)\big)=0,
\end{eqnarray*}
which implies that \eqref{eq:rep1} holds.

Similarly , we have
\begin{eqnarray*}
&&\varrho(x_1,\cdots,x_{n-1},\{y_1,\cdots,y_n,y_{n+1}\})-\sum\limits_{i=1}^{n+1}(-1)^{n+1-i}\varrho(y_1,\cdots,\hat{y_i},\cdots,y_{n+1})\varrho(x_1,\cdots,x_{n-1},y_i)\\
&=&\sum\limits_{k=1}^{n-1}\sum\limits_{j=1}^{n+1}(-1)^{k+j}f(x_k)f(y_j)\rho(x_1,\cdots,\hat{x_k},\cdots,x_{n-1},
[y_1,\cdots,\hat{y_j},\cdots,y_{n+1}]_\g)\\
&&-\sum\limits_{j=1}^{n+1}\sum\limits_{k=1}^{n-1}(-1)^{n+k-j}\big(\sum\limits_{i=1}^{j-1}(-1)^{i}f(y_j)f(x_k)\rho(y_1,\cdots,\hat{y_i},\cdots,\hat{y_j}
,\cdots,y_{n+1})\rho(x_1,\cdots,\hat{x_k},\cdots,x_{n-1},y_i)\\
&&+\sum\limits_{i=j+1}^{n+1}(-1)^{i+1}f(y_j)f(x_k)\rho(y_1,\cdots,\hat{y_j},\cdots,\hat{y_i}
,\cdots,y_{n+1})\rho(x_1,\cdots,\hat{x_k},\cdots,x_{n-1},y_i)\big)\\
&=&\sum\limits_{k=1}^{n-1}\sum\limits_{j=1}^{n+1}(-1)^{k+j}f(y_j)f(x_k)\big(\rho(x_1,\cdots,\hat{x_k},\cdots,x_{n-1},
[y_1,\cdots,\hat{y_j},\cdots,y_{n+1}]_\g)\\
&&-\sum\limits_{i=1}^{j-1}(-1)^{n-i}\rho(y_1,\cdots,\hat{y_i},\cdots,\hat{y_j}
,\cdots,y_{n+1})\rho(x_1,\cdots,\hat{x_k},\cdots,x_{n-1},y_i) \\
&&-\sum\limits_{i=j+1}^{n+1}(-1)^{n-i+1}\rho(y_1,\cdots,\hat{y_j},\cdots,\hat{y_i}
,\cdots,y_{n+1})\rho(x_1,\cdots,\hat{x_k},\cdots,x_{n-1},y_i) \big)=0,
\end{eqnarray*}
which implies that \eqref{eq:rep2} holds. Thus $(\g_f;\varrho)$ is an $(n+1)$-\LP pair.
\end{proof}
Denote the cochain complexes of $n$-Lie algebra $\g$ associated to the representation $\rho$ and $(n+1)$-Lie algebra $\g_f$ associated to the representation $\varrho$ by $(\oplus^{+\infty}_{m=1}C^{m}_{\nl}(\g;V),\partial_\rho)$ and $(\oplus^{+\infty}_{m=1}C^{m}_{ {(n+1)-\rm{Lie}}}(\g_f;V),\tilde{\partial}_\varrho)$, respectively.
\begin{thm}\label{theooo}
Let $(\frak g,[-,\cdots,-];\rho)$ be an $n$-\LP pair and $(\g_f;\varrho)$ be the corresponding $(n+1)$-\LP pair. For $P\in C^{m+1}(\g;V)~(m\geq 1)$, define $\widetilde{P}(\frkX_1,\cdots,\frkX_m,x)$ by
\begin{eqnarray}\label{eq:lift-cocycle}
&&\widetilde{P}(\frkX_1,\cdots,\frkX_m,x):=\sum\limits_{i_1,\cdots,i_m=1}^{n}(-1)^{i_1+\cdots+i_m-m}f(x_1^{i_1})\cdots f(x_m^{i_m})
P(\frkX^{\widehat{i_1}}_1,\cdots,\frkX^{\widehat{i_m}}_m,x)\\
\nonumber&&+\sum\limits_{i_1,\cdots,i_{m-1}=1}^{n}(-1)^{i_1+\cdots+i_{m-1}+n+1-m}f(x_1^{i_1})\cdots f(x_{m-1}^{i_{m-1}})f(x)
P(\frkX^{\widehat{i_1}}_1,\cdots,\frkX^{\widehat{i_{m-1}}}_{m-1},\frkX_m),
\end{eqnarray}
where $\frkX_j:=x_j^1\wedge\cdots \wedge x_j^n$ and $\frkX_j^{\widehat{{i_j}}}:=x_j^1\wedge \cdots \wedge\widehat{x_j^{i_j}}\cdots\wedge x_j^n$. Then $\tilde{P}\in {C^{m+1}}(\g_f;V)$, i.e., $\tilde{P}$ is an $(m+1)$-cochain of the $n$-Lie algebra $\g_f$. Thus we obtain a well-defined linear map
$$\Phi:\oplus^{+\infty}_{m=1}C^{m}_{\nl}(\g;V)\longrightarrow \oplus^{+\infty}_{m=1}C^{m}_{ {(n+1)-\rm{Lie}}}(\g_f;V)$$
 defined by
 \begin{eqnarray}
   \Phi(P)=
   \begin{cases}
     \widetilde{P}, & \forall~P\in C^{m}(\g,V)~(m\geq2),\\
    P,& \forall~P\in C^{1}(\g,V).
     \end{cases}
\end{eqnarray}

Furthermore, we also have $\tilde{\partial}_\varrho \circ \Phi=\Phi\circ\partial_\rho $, i.e., $\Phi$ is a chain map between $(\oplus^{+\infty}_{m=1}C^{m}_{\nl}(\g;V),\partial_\rho)$ and $(\oplus^{+\infty}_{m=1}C^{m}_{ {(n+1)-\rm{Lie}}}(\g_f;V),\tilde{\partial}_\varrho)$. Thus $\Phi$ induces a map $$\Phi_\ast:\oplus^{+\infty}_{m=1}H^m_{\nl}(\g;V)\longrightarrow\oplus^{+\infty}_{m=1}H^m_{ {(n+1)-\rm{Lie}}}(\g_f;V) $$
given by
$$\Phi_\ast([P])=[\Phi(P)],\quad \forall~\in [P]\in H^m_{\nl}(\g;V).$$
\end{thm}
\begin{proof}
  It follows by a long but straightforward computation.
\end{proof}

\begin{pro}\label{pro:rRN-operator}
Let $T:V\rightarrow \g$ be a relative Rota-Baxter operator on an $n$-\LP pair $(\g,[-,\cdots,-]_{\g};\rho)$. Then $T$ is also a relative Rota-Baxter operator on an $(n+1)$-\LP pair $(\g_f;\varrho)$.
\end{pro}
\begin{proof}
By that the fact that $T:V\rightarrow \g$ is a relative Rota-Baxter operator on the $n$-\LP pair $(\g;\rho)$, we have
\begin{eqnarray*}
&&\{Tu_1,\cdots,Tu_{n+1}\}\\
&=&\sum\limits_{i=1}^{n+1}(-1)^{i-1}f(Tu_i)[Tu_1,\cdots,\widehat{Tu_i},\cdots,Tu_{n+1}]_\g\\
&=&\sum\limits_{i=1}^{n+1}(-1)^{i-1}f(Tu_i)\sum\limits_{j=1}^{i-1}(-1)^{n-j}T\rho(Tu_1,\cdots,\widehat{Tu_j},\cdots,\widehat{Tu_i},\cdots,Tu_{n+1})(u_j)\\
&&+\sum\limits_{i=1}^{n+1}(-1)^{i-1}f(Tu_i)\sum\limits_{j=i+1}^{n+1}(-1)^{n-j+1}T\rho(Tu_1,\cdots,\widehat{Tu_i},\cdots,\widehat{Tu_j},\cdots,Tu_{n+1})(u_j)\\
&=&T\sum\limits_{j=1}^{n+1}(-1)^{n+1-j}\big(\sum\limits_{i=j+1}^{n+1}(-1)^if(Tu_i)\rho(Tu_1,\cdots,\widehat{Tu_j},\cdots,\widehat{Tu_i},\cdots,Tu_{n+1})(u_j)\\
&&+\sum\limits_{i=1}^{j-1}(-1)^{i-1}f(Tu_i)\rho(Tu_1,\cdots,\widehat{Tu_i},\cdots,\widehat{Tu_j},\cdots,Tu_{n+1})(u_j)\big)\\
&=&T\sum\limits_{j=1}^{n+1}(-1)^{n+1-j}\varrho(Tu_1,\cdots,\widehat{Tu_j},\cdots,Tu_{n+1}).
\end{eqnarray*}
This implies that $T$ is a relative Rota-Baxter operator on an $(n+1)$-\LP pair $(\g_f;\varrho)$.
\end{proof}
If $T:V\rightarrow \g$ is a relative Rota-Baxter operator on an $n$-\LP pair $(\g,[-,\cdots,-]_{\g};\rho)$, by Proposition \ref{pro:rRN-operator}, $T$ is a relative Rota-Baxter operator on an $(n+1)$-\LP pair $(\g_f;\varrho)$. Thus by Proposition \ref{pro6.5},  $(V,\Courant{-,\cdots,-}_T)$ is an $(n+1)$-Lie algebra, where the bracket $\Courant{-,\cdots,-}_T$ is given by
\begin{eqnarray}
\Courant{u_1,\cdots,u_n,u_{n+1}}_{T}&=&\sum\limits^{n+1}_{i=1}(-1)^{n+1-i}\varrho(Tu_1,\cdots,\widehat{Tu_i},\cdots,Tu_n,Tu_{n+1})(u_i),\quad \forall~u_1,\cdots,u_{n+1}\in V.
\end{eqnarray}
Furthermore, by a direct calculation, we have
\begin{equation}
\Courant{u_1,\cdots,u_n,u_{n+1}}_{T}=\sum\limits_{i=1}^{n+1}(-1)^{i-1}(f\circ T)(u_i)[u_1,\cdots,\hat{u_i},\cdots,u_n,u_{n+1}]_T,\quad \forall~u_1,\cdots,u_n\in V.
  \end{equation}
By the fact that $T$ is a relative Rota-Baxter operator on an $n$-\LP pair $(\g,[-,\cdots,-]_{\g};\rho)$ and $f([x_1,\cdots,x_n]_\g)=0$ for $x_1,\cdots,x_n\in\g$, we have
\begin{eqnarray}
  (f\circ T)[u_1,\cdots,u_{n}]_T=f([T u_1,\cdots,T u_n]_\g)=0.
\end{eqnarray}
Thus we have
\begin{pro}\label{pro:construction by RB-operators}
The $(n+1)$-Lie algebra $(V,\Courant{-,\cdots,-}_T)$ is just the $(n+1)$-Lie algebra constructed by the $n$-Lie algebra $(V,[-,\cdots,-]_T)$ with $ f\circ T\in V^*$ satisfying $(f\circ T)[u_1,\cdots,u_{n}]_T=0$ through the way given in Lemma \ref{n-n+1}.
\end{pro}

Denote the cochain complexes of the relative Rota-Baxter operator $T$ on both  the $n$-\LP pair $(\frak g;\rho)$  and the $(n+1)$-\LP pair $(\frak g_f;\varrho)$ by $(\oplus^{+\infty}_{m=0}C^{m}_T(V;\g),d)$ and $(\oplus^{+\infty}_{m=0}C^{m}_{T}(V;\g_f),\tilde{d})$, respectively.
\begin{pro}
With the above notations, there is  a chain map
	$$\Phi:(\oplus^{+\infty}_{m=0}C^{m}_T(V;\g),d)\longrightarrow (\oplus^{+\infty}_{m=0}C^{m}_{T}(V;\g_f),\tilde{d})$$
defined by
\begin{eqnarray}
   \Phi(P)=
   \begin{cases}
     \widetilde{P}, & \forall~P\in C^{m}(V,\g)~(m\geq2),\\
    P,& \forall~P\in C^{1}(V,\g),\\
    P\wedge x_0, & \forall~P\in \wedge^{n-1} \g,
     \end{cases}
\end{eqnarray}
where $x_0$ is an element in the center of the semi-direct product $n$-Lie algebra $\g\ltimes_{\rho} V$ and $\widetilde{P}$ is given by
\begin{eqnarray}\label{eq:RB-lift}
&&\widetilde{P}(\frkU_1,\cdots,\frkU_{m-1},v):=\sum\limits_{i_1,\cdots,i_{m-1}=1}^{n}(-1)^{i_1+\cdots+i_{m-1}-m+1}(f\circ T)(u_1^{i_1})\cdots (f\circ T)(u_{m-1}^{i_{m-1}})
P(\frkU^{\widehat{i_1}}_1,\cdots,\frkU^{\widehat{i_{m-1}}}_{m-1},v)\\
\nonumber&&+\sum\limits_{i_1,\cdots,i_{m-2}=1}^{n}(-1)^{i_1+\cdots+i_{m-1}+n-m}(f\circ T)(u_1^{i_1})\cdots (f\circ T)(u_{m-2}^{i_{m-2}})(f\circ T)(v)
P(\frkU^{\widehat{i_1}}_1,\cdots,\frkU^{\widehat{i_{m-2}}}_{m-2},\frkU_{m-1}),
\end{eqnarray}
where $\frkU_j:=u_j^1\wedge\cdots \wedge u_j^n$ and $\frkU_j^{\widehat{{i_j}}}:=u_j^1\wedge \cdots \wedge\widehat{u_j^{i_j}}\cdots\wedge u_j^n$.

Thus $\Phi$ induces a map $$\Phi_\ast:\oplus^{+\infty}_{m=0}\huaH^m_{\nl}(V;\g)\longrightarrow\oplus^{+\infty}_{m=0}\huaH^m_{ {(n+1)-\rm{Lie}}}(V;\g_f) $$
given by
$$\Phi_\ast([P])=[\Phi(P)],\quad \forall~\in [P]\in \huaH^m_{\nl}(V;\g).$$
\end{pro}
\begin{proof}
Note that the cochain complex $(\oplus^{+\infty}_{m=1}C^{m}_T(V;\g),d)$ of the relative Rota-Baxter operator $T$ on  the $n$-\LP pair $(\frak g;\rho)$  is just the cochain complexes of the $n$-Lie algebra $(V,[-,\cdots,-]_T)$ associated to the representation $(\g;\rho_T)$ and the cochain complexes $(\oplus^{+\infty}_{m=1}C^{m}_T(V;\g_f),\tilde{d})$ of the relative Rota-Baxter operator $T$ on  the $(n+1)$-\LP pair $(\frak g_f;\varrho)$  is just the cochain complexes of the $(n+1)$-Lie algebra $(V,[-,\cdots,-]_T)$ associated to the representation $(\g;\varrho_T)$.  By Proposition \ref{pro:construction by RB-operators},  Eq. \eqref{eq:RB-lift} is just Eq. \eqref{eq:lift-cocycle} by replacing  the $n$-Lie algebra $(\g,[-,\cdots,-]_\g)$ and $f\in \g^*$ with the  $n$-Lie algebra $(V,[-,\cdots,-]_T)$ and  $ f\circ T\in V^*$. Thus by Theorem \ref{theooo}, 	$$\Phi:(\oplus^{+\infty}_{m=1}C^{m}_T(V;\g),d)\longrightarrow (\oplus^{+\infty}_{m=1}C^{m}_{T}(V;\g_f),\tilde{d})$$
given by \eqref{eq:RB-lift} is a chain map.

By the fact that $x_0$ is an element in the center of the semi-direct product $n$-Lie algebra $\g\ltimes_{\rho} V$, it is straightforward to check that
$$\Phi (d \frkX)=\tilde{d}\Phi (\frkX),\quad \forall~\frkX\in \wedge^{n-1}\g.$$
Thus $\Phi$ is a chain map. The rest is direct.
\end{proof}

\noindent
{\bf Acknowledgements. } This research is supported by NSFC (11771410,11801066,11901501). We give our warmest thanks to Yunhe Sheng and Rong Tang for very useful comments and discussions.

\end{document}